\documentclass[a4paper]{article}
\usepackage{eurosym}
\usepackage[pages=all, color=black, position={current page.south}, placement=bottom, scale=1, opacity=1, vshift=5mm]{background}
\usepackage[margin=1in]{geometry}
\usepackage{amsmath}
\usepackage{amsthm}
\usepackage{amssymb}
\usepackage{amssymb}
\usepackage{amsfonts}
\usepackage{amsmath}
\usepackage{amssymb}
\usepackage{amsfonts}
\usepackage{amsmath,amsthm,subfig,subfloat,graphics}
\usepackage{amsmath,subfig,subfloat,graphics}
\usepackage{graphicx}
\usepackage{caption}
\usepackage[T1]{fontenc}
\usepackage[utf8]{inputenc}
\usepackage{hyperref}
\usepackage[sort&compress,numbers,square]{natbib}
\usepackage{graphicx, color}
\usepackage{algorithm, algpseudocode}
\usepackage{mathrsfs}
\usepackage{lipsum}
\usepackage{rotating} 
\usepackage{float}
\usepackage{changepage}
\usepackage{pdflscape}
\usepackage{tikz}
\usetikzlibrary{positioning}

\SetBgContents{
}      
\hypersetup{
unicode,
colorlinks,
breaklinks,
urlcolor=cyan, 
linkcolor=blue, 
pdfauthor={Author One, Author Two, Author Three},
pdftitle={A simple article template},
pdfsubject={A simple article template},
pdfkeywords={article, template, simple},
pdfproducer={LaTeX},
pdfcreator={pdflatex}
}
\theoremstyle{plain}
\newtheorem{theorem}{Theorem}

\theoremstyle{definition}

\begin{document}

\title{Optimizing Parallel Schemes with Lyapunov Exponents and kNN-LLE Estimation}
\author{Mudassir Shams$^1$$^,$$^2$$^,$$^*$, Andrei Velichko$^3$, Bruno Carpentieri$^2$ }
\date{$^1$Department of Mathematics, Faculty of Arts and Science, Balikesir
University, Balikesir 10145, Turkey\\
$^2$Faculty of Engineering, Free University of Bozen-Bolzano, 39100,
Bolzano, Italy\\
$^3$Institute of Physics and Technology, Petrozavodsk State University, 185910 Petrozavodsk, Russia \\
\texttt{mudassir.shams@balikesir.edu.tr;velichko.th@gmail.com}\\
\texttt{bruno.carpentieri@unibz.it}\\
\texttt{$^*$Correspondence:Mudassir Shams: Email:mudassir.shams@balikesir.edu.tr}\\
[2ex]\today
}
\maketitle

\begin{abstract}
Inverse parallel schemes remain indispensable tools for computing the roots of nonlinear systems, yet their dynamical behavior can be unexpectedly rich, ranging from strong contraction to oscillatory or chaotic transients depending on the choice of algorithmic parameters and initial states. A unified analytical—data-driven methodology for identifying, measuring, and reducing such instabilities in a family of uni-parametric inverse parallel solvers is presented in this study. On the theoretical side, we derive stability and bifurcation characterizations of the underlying iterative maps, identifying parameter regions associated with periodic or chaotic behavior. On the computational side, we introduce a micro-series pipeline based on kNN-driven estimation of the local largest Lyapunov exponent (LLE), applied to scalar time series derived from solver trajectories. The resulting sliding-window Lyapunov profiles provide fine-grained, real-time diagnostics of contractive or unstable phases and reveal transient behaviors not captured by coarse linearized analysis. Leveraging this correspondence, we introduce a Lyapunov-informed parameter selection strategy that identifies solver settings associated with stable behavior, particularly when the estimated LLE indicates persistent instability. Comprehensive experiments on ensembles of perturbed initial guesses demonstrate close agreement between the theoretical stability diagrams and empirical Lyapunov profiles, and show that the proposed adaptive mechanism significantly improves robustness. The study establishes micro-series Lyapunov analysis as a practical, interpretable tool for constructing self-stabilizing root-finding schemes and opens avenues for extending such diagnostics to higher-dimensional or noise-contaminated problems.

\noindent\textbf{Keywords:} {Uni-parametric inverse parallel solvers; Fractional-order scheme; Local convergence analysis; Micro-series Lyapunov analysis; kNN–LLE parameter estimation; Error-norm analysis}
\end{abstract}

\section{Introduction}

Numerical analysis and scientific computing rely on iterative methods for solving nonlinear problems as core tools~\cite{1,2,3}. Iterative approaches~\cite{4} serve as the foundation for many engineering, physics, and applied mathematics applications. Traditionally, researchers have focused on demonstrating technique convergence, defining maximum performance limitations, and confirming functionality under ideal conditions. Better root-finding methods are required as the complexity of computational models in automated systems working under high-pressure situations increases. These iterative schemes should be faster, more precise, and durable, capable of stabilizing themselves across a wide range of nonlinear problems and starting points~\cite{5}.
In practice, many advanced iterative schemes—such as inverse parallel methods~\cite{6} or modified Weierstrass-type algorithms~\cite{7} equipped with tunable uni-parametric coefficients—are designed to accelerate convergence and improve efficiency. Yet, these enhancements often introduce intricate dynamical behaviors. The behavior of these methods depends on parameter selection and initial guess accuracy because they produce different results including short-term fluctuations and stable periodic patterns and chaotic behavior. The theoretical concepts of these systems directly affect solver stability and error sensitivity and automated computational pipeline reliability. The solver's performance becomes unpredictable when it encounters chaotic or unstable regimes which results in either extended computation times or complete failure to generate solutions~\cite{8}.\\

The process of selecting appropriate parameters for iterative methods to find roots of nonlinear equations determines how fast the methods converge while it also controls their stability and overall reliability. Research from recent studies~\cite{9,10} shows researchers should not select parameters randomly because complex system behaviors emerge when scientists analyze higher-order and multi-step schemes through dynamical systems theory. The global convergence behavior of dynamical systems becomes clear when using dynamical and parameter planes which display attraction basins and fast convergence areas and parameter settings that lead to system failure or unstable behavior or numerical calculation issues~\cite{11}. The study of iterative map behavior demands advanced diagnostic tools beyond two-dimensional visualizations which include bifurcation diagramsand chaos indicators~\cite{12}  and Lyapunov exponents and periodic orbit detection methods~\cite{13}. The analyses show how systems move from steady state behavior to period-doubling progression which leads to complex periodic patterns and then chaotic behavior. The theoretical nature of these dynamical signatures fades because they help determine optimal parameter values through their ability to identify stable contractive behavior ranges and unstable sensitivity regions.\\
To address these challenges, this paper presents a comprehensive framework that integrates both theoretical and data-driven approaches for diagnosing and mitigating dynamical instabilities in inverse parallel root-finding algorithms. The proposed methodology combines 
\begin{itemize}
    \renewcommand{\labelitemi}{--}
	\item a rigorous analytical investigation of bifurcations and stability within the family of uni-parametric iterative maps—mapping out regions of stability, periodicity, and chaos in parameter space—with 
    \item a real-time, data-driven diagnostic pipeline based on Lyapunov exponent estimation. Specifically, the data-driven component leverages time series data from solver iterates (such as step sizes and residual norms) and applies a K-nearest neighbors (kNN) micro-series estimator to compute sliding-window estimates of the local largest Lyapunov exponent (LLE)~\cite{14}. 
\end{itemize}
Taken together, dynamical planes, parameter-space exploration, and chaos-based diagnostics form a unified framework for selecting both effective starting points and optimal parameter configurations for modern iterative schemes~\cite{15}. This integrated viewpoint ensures that algorithm designers do not rely solely on local convergence theorems, but instead incorporate global, data-driven dynamical insights that guarantee reliability across a broad range of nonlinear problems. By analyzing Lyapunov profiles in conjunction with the method's recognized stability features, researchers can compare theoretical predictions to empirical performance. Through this process researchers can identify complex or short-term instabilities which analytical methods fail to detect. The system enables users to create adaptive feedback mechanisms through its support structure. These mechanisms adjust solver parameters automatically to stabilize the system when continuous positive LLEs indicate instability.\\

The global behavior~\cite{14} of iterative methods becomes accessible through analytical stability and bifurcation analysis because these techniques identify when methods show contractive behavior and when they develop bifurcations and chaotic behavior. These methods depend on linearized models and local approximations which restrict their application. The methods fail to detect temporary computational results and finite sample sizes and high-dimensional effects which appear during actual computations. Data-driven dynamical diagnostics~\cite{16} which use time-resolved Lyapunov exponent estimation methods~\cite{17} achieve operational evaluation through the measurement of actual numerical trajectories that computer simulations produce. The dual method enables to verify and enhance theoretical models while it identifies unstable system states and early warning signs at the beginning of failure events. The integration of these methods enables researchers to create adaptive root-finding solvers which autonomously stabilize themselves for producing dependable solutions to complex problems in real-time.\\ 

The study aims to investigate the dynamics of inverse parallel schemes by characterizing transitions between stable, periodic, and chaotic regimes. It further seeks to compare analytical bifurcation diagrams with empirical largest Lyapunov exponent (LLE) curves to assess the reliability of short-horizon LLE values as early-warning indicators, while also designing and testing an adaptive control mechanism to stabilize the solver’s behavior.

This work makes the following contributions:

\begin{itemize}
    \renewcommand{\labelitemi}{--}
	\item Development and analysis of the inverse parallel schemes for solving nonlinear equations
    \item A reproducible micro-series pipeline for sliding-window LLE estimation using solver iterates.
    \item Empirical validation showing data-driven Lyapunov profiles align with analytical stability regions while revealing fine transient instabilities.
    \item Operational diagnostics: short-horizon $\lambda_1(t)$ acts as an early-warning indicator.
    \item A simple adaptive control rule that modifies uni-parametric coefficients $(\alpha)$ when $\lambda_1$ becomes persistently positive.
    \item A fully open experimental suite with generators, raw matrices, and diagnostic reports.
\end{itemize}
The rest of the study is organized as follows:
Section~2 reviews related work, introduces the modified inverse parallel scheme and its uni-parametric space, and outlines the micro-series kNN estimator and experimental protocol. It also presents the Lyapunov profiles and their comparison with the analytical stability diagrams. Section~3 develops and evaluates the adaptive feedback rule. Section~4 concludes and outlines future extensions.

\section{Development, Methodology, and Implementation of Lyapunov Profiling}

In this section, we first propose the inverse parallel schemes and then provide a detailed exposition of our methodology, encompassing the mathematical formulation of the iterative scheme, the rationale for employing Lyapunov analysis, the micro-series kNN-based Lyapunov estimator, the experimental setup, and the structure of the numerical results.

After the impossibility theorems~\cite{18}, parallel or simultaneous methods are employed to approximate all roots of nonlinear equations, particularly polynomials, due to their global convergence, enhanced stability, reliability, and suitability for parallel architectures. Its capability to compute distinct and all multiple roots concurrently renders it especially effective for high-degree polynomials and large-scale nonlinear problems~\cite{19}. The method’s simplicity, global convergence properties, and independence from derivative evaluations make it a robust and widely adopted tool in scientific and engineering computations. Furthermore, its inherently parallel structure allows seamless implementation on modern multicore processors, GPU architectures, and distributed computing platforms, thereby facilitating fast, reliable, and scalable root-finding performance across a broad range of applied science and engineering applications.

Among the classical schemes, the well-known Weierstrass--Durand--Kerner~\cite{20} with local quadratic convergence is defined by
\begin{equation}
x_{i}^{[h+1]}
    =x_{i}^{[h]}
    -\frac{f\!\left(x_{i}^{[h]}\right)}
    {\displaystyle\prod_{j\neq i}^{n}\left(x_{i}^{[h]} - x_{j}^{[h]}\right)}.
\end{equation}

Nourein's method~\cite{21} is given by
\begin{equation}
x_{i}^{[h+1]}
   = x_{i}^{[h]}
   -\frac{\mathcal{P}\!\left(x_{i}^{[h]}\right)}
   {\displaystyle 1+
   \sum_{j\neq i}^{n}
   \left(
   \frac{\mathcal{P}(x_{j}^{[h]})}
   {\,x_{i}^{[h]}
    -\mathcal{P}(x_{j}^{[h]})
    -x_{j}^{[h]}}
   \right)},
\end{equation}
where
\[
\mathcal{P}(x_{i}^{[h]})
   =\frac{f(x_{i}^{[h]})}
   {\displaystyle\prod_{j\neq i}\left(x_{i}^{[h]} - x_{j}^{[h]}\right)}. 
\]

Zhang et al.~\cite{22} introduced the fifth-order method (denoted by ZHM):
\begin{equation}
x_{i}^{[h+1]}
  = x_{i}^{[h]}
  -\frac{\mathcal{P}(x_{i}^{[h]})}
   {1+\mathcal{G}_i^{[\ast]}(x_{i}^{[h]})
    +\sqrt{\mathcal{K}_{1.1}}},
    \label{1d}
\end{equation}
where
\[
\mathcal{K}_{1.1}
=
1+\bigl(\mathcal{G}_i^{[\ast]}(x_{i}^{[h]})\bigr)^2
+
4\mathcal{P}(x_{i}^{[h]})
\sum_{j\neq i}^{n}
\!\left(
\frac{\mathcal{P}(x_{j}^{[h]})}
{(x_{i}^{[h]}-x_{j}^{[h]})
 (x_{i}^{[h]}-\mathcal{P}(x_{j}^{[h]})-x_{j}^{[h]})}
\right),
\]
and
\[
\mathcal{G}_i^{[\ast]}(x_{i}^{[h]})
    =\frac{\mathcal{P}(x_{i}^{[h]})}
    {\displaystyle\sum_{j\neq i}^{n}(x_{i}^{[h]} - x_{j}^{[h]})}.
\]

Several researchers have also proposed multi-step simultaneous root-finding schemes, including Petković et al.~\cite{23}, Proinov~\cite{24}, Mir et al.~\cite{25}, Nedzhibov~\cite{26}, Cordero et al.~\cite{27}, and many others (see, e.g., \cite{28,29,30} and the references therein).

\bigskip

\noindent\textbf{Proposed Inverse Parallel Fractional Scheme.}
We proposed the following inverse parallel iterative method (INVM$^{\alpha}$):
\begin{equation}
x_{i}^{[h+1]}
 =
\frac{
\left(x_{i}^{[h]}\right)^{2}
\displaystyle\prod_{j\neq i}
^{n}\left(
x_{i}^{[h]}
 -y_{j}^{[h]}
 -\Bigl[\Gamma(\beta+1)\frac{f(x_{j}^{[h]})}{D^C f(x_{j}^{[h]})}
 \bigl(1+\frac{2(f(x_j^{[h]})/f(y_j^{[h]}))}{1+\alpha(f(x_j^{[h]})/f(y_j^{[h]}))}\bigr)
 \Bigr]^{1/\beta}
\right)}
{
x_{i}^{[h]}
\displaystyle\prod_{j\neq i}^{n}
\left(
x_{i}^{[h]}
 -y_{j}^{[h]}
 -\Bigl[\Gamma(\beta+1)\frac{f(x_{j}^{[h]})}{D^C f(x_{j}^{[h]})}
 \bigl(1+\frac{2(f(x_j^{[h]})/f(y_j^{[h]}))}{1+\alpha(f(x_j^{[h]})/f(y_j^{[h]}))}\bigr)
 \Bigr]^{1/\beta}
\right)
 + f(x_{i}^{[h]})
}.
\label{1ss}
\end{equation}

Here,
\[
y_{j}^{[h]}
   = x_{j}^{[h]}
     -\Bigl[\Gamma(\beta+1)\frac{f(x_{j}^{[h]})}{D^{C}f(x_{j}^{[h]})}\Bigr]^{1/\beta}
\]
and $\alpha \in R$.
The scheme (\ref{1ss}) can be rewritten in the compact parallel-like form
\begin{equation}
x_{i}^{[h+1]}
   = x_{i}^{[h]}
     -\frac{\mathcal{P}^{[\ast]}(x_{i}^{[h]})}
     {1+\dfrac{\mathcal{P}^{[\ast]}(x_{i}^{[h]})}{x_{i}^{[h]}}},
\end{equation}
where
\begin{equation}
\mathcal{P}^{[\ast]}(x_{i}^{[h]})
   = \frac{f(x_{i}^{[h]})}
   {\displaystyle\prod_{j\neq i}^{n}\left(x_{i}^{[h]} - z_{j}^{[h]}\right)}
\end{equation}
and 

\begin{equation}
\left\{
\begin{aligned}
y_{j}^{[h]} 
&= x_{j}^{[h]} 
   - \Biggl(\Gamma(\beta+1)\,
     \frac{f(x_{j}^{[h]})}{D^{C}f(x_{j}^{[h]})} \Biggr)^{1/\beta},\\[4pt]
z_{j}^{[h]} 
&= y_{j}^{[h]} 
   - \Biggl[ \Gamma(\beta+1)\,
     \frac{f(x_{j}^{[h]})}{D^{C}f(x_{j}^{[h]})} 
     \Biggl( 1 + \frac{2\,\dfrac{f(x_j^{[h]})}{f(y_j^{[h]})}}
     {1 + \alpha\,\dfrac{f(x_j^{[h]})}{f(y_j^{[h]})}} \Biggr)
   \Biggr]^{1/\beta}
\end{aligned}
\right.
\label{2s}
\end{equation}

and \(D^{C}\) is the Caputo fractional derivative of order \(\beta\in(0,1]\), defined as~\cite{31}

\begin{equation}
D^{C}_{}f(x)
= \frac{1}{\Gamma(1-\beta)}
  \int_{0}^{x} \frac{f'(s)}{(x-s)^{\beta}}\, ds,
\label{3t}
\end{equation}

where \(\Gamma(\cdot)\) denotes the Gamma function~\cite{32}, given for \(\gamma>0\) by

\begin{equation}
\Gamma(\gamma)
= \int_{0}^{\infty} e^{-u}\, u^{\gamma-1}\, du.
\label{4t}
\end{equation}

The operator \(D^{C}_{}\) acts on continuously differentiable functions  
\(f \in C^{1}[0,T]\) and reduces to the classical first-order derivative when \(\beta=1\).

The iterative scheme defined in \((\ref{2s})\), which possesses a fractional order of convergence \(3\beta+1\), satisfies the following error relations.  
Let the errors at the \(h\)-th iteration be
\[
\varepsilon_{j}^{[h]} = x_{j}^{[h]} - \zeta_j,
\qquad
\varepsilon_{j}^{[h+1]} = z_{j}^{[h]} - \zeta_j,
\]
where \(\zeta_j\) denotes the exact roots.  

Under these definitions, the fractional-order scheme (\ref{2s}) admits the local error expansion
\[
\varepsilon_i^{(\ast)}
   = \Bigg(
      A_1 - A_2 - A_3
      + 2A_4
      - 2A_5
      + A_6
      - A_7 + A_8 - A_9 - A_{10} + A_{11}
     \Bigg)
     \varepsilon_i^{3\beta+1}
     + O(\varepsilon_i^{4\beta+1}),
\label{3s}
\]
where the coefficients $A_1,\ldots,A_{11}$ are
\[
\begin{aligned}
A_1 &= \frac{(\tfrac{2}{\beta})^{2}
       \Gamma(\beta+\tfrac12)\check{c}^{3/2}}
       {\beta^{2}\Gamma(\beta)^{2}2\pi},
\qquad
A_2 = \frac{2\check{c}^{3/2}
       (\tfrac{2}{\beta})^{6}
       \Gamma(\beta+\tfrac12)^{3}}
       {\beta^{3}\Gamma(\beta)^{3}2\pi^{3/2}},
\\[4pt]
A_3 &= \frac{(\tfrac{2}{\beta})
       \Gamma(\beta+\tfrac12)
       \check{c}_3\check{c}_2}
       {\beta\Gamma(\beta)\sqrt{\pi}},
\qquad
A_4 = \frac{\check{c}_3\check{c}_2
        (\tfrac{3}{\beta})^{3}
        \sqrt{3}\,
        \Gamma(\beta+\tfrac13)
        \Gamma(\beta+\tfrac23)}
        {\beta^{2}\Gamma(\beta)^{2}\pi},
\\[4pt]
A_5 &= \frac{\check{c}_3\check{c}_2}{2}
       \frac{(\tfrac{3}{\beta})^{3}\sqrt{3}\,
       \Gamma(\beta+\tfrac13)
       \Gamma(\beta+\tfrac23)}
       {\beta\Gamma(\beta)\sqrt{\pi}},
\\[4pt]
A_6 &= \tfrac12\varsigma''(0)\check{c}^{3/2},
\qquad
A_7 = \check{c}^{3/2},
\\[4pt]
A_8 &= \frac{3}{2}
       \frac{(\tfrac{2}{\beta})^{2}
       \Gamma(\beta+\tfrac12)\check{c}^{3/2}}
       {\beta^{2}\Gamma(\beta)^{2}2\pi},
\qquad
A_9 = \frac12
      \frac{(\tfrac{2}{\beta})^{6}
      \Gamma(\beta+\tfrac12)^{3}
      \check{c}^{3/2}}
      {\beta^{3}\Gamma(\beta)^{3}\pi^{3/2}},
\\[4pt]
A_{10} &= \frac{3}{2}
         \frac{(\tfrac{2}{\beta})^{2}
         \Gamma(\beta+\tfrac12)\check{c}^{3/2}}
         {\beta\Gamma(\beta)\sqrt{\pi}},
\qquad
A_{11} = 2
         \frac{(\tfrac{2}{\beta})
         \Gamma(\beta+\tfrac12)\check{c}^{3/2}}
         {\beta\Gamma(\beta)\sqrt{\pi}}.
\end{aligned}
\]

The constants $c_{\gamma}$ are defined as
\[
c_{\gamma}
   = \frac{\Gamma(\beta+1)}{\Gamma(\gamma\beta+1)}
     \frac{D^{\gamma C}f(\zeta_j)}{D^{C}f(\zeta_j)},
\qquad \gamma\ge 2.
\]
which confirms that the dominant term in the error behaves as \(\varepsilon_i^{3\beta+1}\), establishing the fractional convergence order of the method.

\bigskip
\subsection{Local Convergence Analysis}

We now present the local convergence properties of the proposed fractional-order scheme for simultaneously approximating all roots of a nonlinear equation. The following result establishes the order of convergence under standard assumptions.

\begin{theorem}
Let $\zeta_{1},\ldots,\zeta_{\upsilon}$ be simple roots of a nonlinear equation $f(x)=0$, and assume that the initial approximations 
$x_{1}^{[0]},\ldots,x_{\upsilon}^{[0]}$ are sufficiently close to the corresponding exact roots.  
Then the fractional-order INVM$^{\alpha}$ method defined in (\ref{1ss}) converges with order \(3\beta+2\).
\end{theorem}

\begin{proof}
Let the errors at iteration $h$ be denoted by
\[
\varepsilon_{i}=x_{i}^{[h]}-\zeta_{i},
\qquad 
\varepsilon_{i}'=x_{i}^{[h+1]}-\zeta_{i}.
\]
Then the iterative update can be written as
\begin{equation}
x_{i}^{[h+1]} - \zeta_{i}
=
x_{i}^{[h]} - \zeta_{i}
-
\frac{
\dfrac{f(x_{i}^{[h]})}{\displaystyle \prod_{j \neq i}^{n} (x_{i}^{[h]} - z_{j}^{[h]})}
}{
1 + \dfrac{\mathcal{P}^{[\ast]}(x_{i}^{[h]})}{x_{i}^{[h]}}}
,
\label{25}
\end{equation}
where
\[
\left\{
\begin{aligned}
y_{j}^{[h]} &= x_{j}^{[h]} - 
\Biggl(
\Gamma(\beta+1)\,
\frac{f(x_{j}^{[h]})}{D^{C} f(x_{j}^{[h]})}
\Biggr)^{1/\beta},\\[2mm]
z_{j}^{[h]} &= y_{j}^{[h]} - 
\Biggl[
\Gamma(\beta+1)\,
\frac{f(x_{j}^{[h]})}{D^{C} f(x_{j}^{[h]})}
\Biggl(
1 + 
\frac{2 \dfrac{f(x_{j}^{[h]})}{f(y_{j}^{[h]})}}
{1 + \alpha \dfrac{f(x_{j}^{[h]})}{f(y_{j}^{[h]})}}
\Biggr)
\Biggr]^{1/\beta}.
\end{aligned}
\right.
\]

For simplicity, equation (\ref{25}) can be rewritten in terms of the improved approximations \(z_{j}^{[h]}\) as shown above.

\begin{equation}
x_{i}^{[h+1]}-\zeta_{i}
=
x_{i}^{[h]}-\zeta_{i}
-
\frac{
\dfrac{f(x_{i}^{[h]})}{
\prod\limits_{j\neq i}^{n}(x_{i}^{[h]}-z_{j}^{[h]})
}}
{
1+\dfrac{\mathcal{P}^{[\ast]}(x_{i}^{[h]})}{x_{i}^{[h]}}}.
\label{26}
\end{equation}

Subtracting $\zeta_{i}$ and introducing the error notation yields
\begin{equation}
\varepsilon_{i}'
=
\varepsilon_{i}
-
\frac{
\varepsilon_{i}
\prod\limits_{j\neq i}^{n}
\left(
\frac{x_{i}^{[h]}-\zeta_{j}}{x_{i}^{[h]}-z_{j}^{[h]}}
\right)
}
{
1+\dfrac{\mathcal{P}^{[\ast]}(x_{i}^{[h]})}{x_{i}^{[h]}}}=
\varepsilon_{i}
\left(
1-
\frac{
\prod\limits_{j\neq i}^{n}
\left(
\frac{x_{i}^{[h]}-\zeta_{j}}{x_{i}^{[h]}-z_{j}^{[h]}}
\right)
}
{
1+\dfrac{\mathcal{P}^{[\ast]}(x_{i}^{[h]})}{x_{i}^{[h]}}}
\right),
\label{27}
\end{equation}
and equivalently,
\begin{equation}
\varepsilon_{i}'
=
\varepsilon_{i}
\left(
\frac{
1-\prod\limits_{j\neq i}^{n}
\left(
\frac{x_{i}^{[h]}-\zeta_{j}}{x_{i}^{[h]}-z_{j}^{[h]}}
\right)
+
\dfrac{z_{i}^{[h]}-\zeta_{i}}{x_{i}^{[h]}}
\prod\limits_{j\neq i}^{n}
\left(
\frac{z_{i}^{[h]}-\zeta_{j}}{x_{i}^{[h]}-z_{j}^{[h]}}
\right)
}{
1+\dfrac{\mathcal{P}^{[\ast]}(x_{i}^{[h]})}{x_{i}^{[h]}}}
\right).
\label{30a}
\end{equation}

Since
\begin{equation}
\prod_{j\neq i}^{n}
\left(
\frac{x_{i}^{[h]}-\zeta_{j}}{x_{i}^{[h]}-z_{j}^{[h]}}
\right)
-1
=
\sum_{k\neq i}^{n}
\frac{-\varepsilon_{j}^{\,3\beta+1}}{x_{i}^{[h]}-z_{j}^{[h]}}
\prod_{k\neq i}^{n}
\left(
\frac{z_{i}^{[h]}-\zeta_{k}}{x_{i}^{[h]}-z_{k}^{[h]}}
\right),
\end{equation}
and using the expansion
\begin{equation}
z_{j}^{[h]}-\zeta_{j}=O(\varepsilon_{j}^{\,3\beta+1}),
\end{equation}
we obtain
\begin{equation}
\varepsilon_{i}'
=
\varepsilon_{i}
\left(
\frac{
\displaystyle
\sum_{j\neq i}^{n}
\frac{-\varepsilon_{j}^{\,3\beta+1}}{x_{i}^{[h]}-z_{j}^{[h]}}
\prod_{k\neq i,j}^{n}
\left(
\frac{z_{i}^{[h]}-\zeta_{k}}{x_{i}^{[h]}-z_{k}^{[h]}}
\right)
+
\dfrac{\varepsilon_{i}^{\,3\beta+1}}{x_{i}^{[h]}}
\prod_{j\neq i}^{n}
\left(
\frac{z_{i}^{[h]}-\zeta_{j}}{x_{i}^{[h]}-z_{j}^{[h]}}
\right)
}{
1+\dfrac{\mathcal{P}^{[\ast]}(x_{i}^{[h]})}{x_{i}^{[h]}}}
\right).
\end{equation}

Assuming equal error magnitudes,
\(
|\varepsilon_{i}|=|\varepsilon_{j}|=|\varepsilon|,
\)
we conclude that
\begin{equation}
\varepsilon_{i}'
=
O\!\left(|\varepsilon|^{\,3\beta+2}\right).
\label{30}
\end{equation}

This establishes that the method converges with fractional order \(3\beta+2\), completing the proof.
\qedhere
\end{proof}

\subsection{Mathematical Description of the Method}

The methodology combines classical convergence metrics with data-driven Lyapunov profiling to validate theoretical stability regions, quantify transient instabilities, and demonstrate the effectiveness of the adaptive feedback mechanism. The detailed micro-series LLE analysis ensures reproducibility, robustness, and predictive power for iterative root-finding schemes with tunable parameters.\\

The primary objective of this study is to construct and analyze a family of
fractional–order Weierstrass-type parallel iterative schemes for computing 
all roots of a nonlinear polynomial equation
\[
f(x)=0.
\]
Let $x_{i}^{[k]}$ denote the $i$-th root approximation at iteration $k$.  
Incorporating a fractional correction factor governed by a tunable order 
parameter $\vartheta$, the proposed update rule is defined as
\[
x_{i}^{[h+1]}
 =
\frac{
\left(x_{i}^{[h]}\right)^{2}
\displaystyle\prod_{j\neq i}
^{n}\left(
x_{i}^{[h]}
 -y_{j}^{[h]}
 -\Bigl[\Gamma(\beta+1)\frac{f(x_{j}^{[h]})}{D^C f(x_{j}^{[h]})}
 \bigl(1+\frac{2(f(x_j^{[h]})/f(y_j^{[h]}))}{1+\alpha(f(x_j^{[h]})/f(y_j^{[h]}))}\bigr)
 \Bigr]^{1/\beta}
\right)}
{
x_{i}^{[h]}
\displaystyle\prod_{j\neq i}^{n}
\left(
x_{i}^{[h]}
 -y_{j}^{[h]}
 -\Bigl[\Gamma(\beta+1)\frac{f(x_{j}^{[h]})}{D^C f(x_{j}^{[h]})}
 \bigl(1+\frac{2(f(x_j^{[h]})/f(y_j^{[h]}))}{1+\alpha(f(x_j^{[h]})/f(y_j^{[h]}))}\bigr)
 \Bigr]^{1/\beta}
\right)
 + f(x_{i}^{[h]})
}.
\label{1s}
\]

The multiplicative correction term involving $\alpha$ allows 
the scheme to adjust its convergence behavior dynamically.  
Unlike classical Weierstrass or Durand–Kerner schemes, the proposed method 
retains stability even for clustered or multiple roots due to the fractional 
regularization embedded in the denominator.

To guarantee robust convergence, appropriate selections of $\alpha$ and 
the initial approximations $\{x_{i}^{[0]}\}$ are crucial.  
To determine optimal parameter ranges, we employ a combination of 
dynamical-plane analysis, bifurcation exploration, and Lyapunov-based 
sensitivity diagnostics, as detailed in the following subsections.

\subsection{Data-Driven Lyapunov Profiling for the Parallel Scheme}

The purpose of this subsection is to explain how we convert raw trajectories
of the fractional inverse parallel scheme INVM$^{\alpha}$ into Lyapunov-based
diagnostics that can be used for parameter tuning. Given an ensemble of solver
runs for several values of $\alpha$, the pipeline produces sliding-window
estimates of the local largest Lyapunov exponent (LLE) and then uses these
profiles to identify parameter regions that are both stable and fast convergent.

Conceptually, the pipeline consists of the following stages (see
Figure~\ref{FL}):

\begin{enumerate}
    \item \textbf{Ensemble generation.} Construct ensembles of initial guesses
    around selected base vectors (different ``cases'') and fix a grid of
    candidate parameter values $\alpha$.
    \item \textbf{Solver runs and observables.} For each case and each
    $\alpha$ run the parallel scheme INVM$^{\alpha}$ and record scalar
    observables at every iteration, namely the step norms $\|s_k\|_2$ and
    the residual norms $\|r_k\|_2$.
    \item \textbf{Micro-series batching.} From the resulting time-series
    matrices build short overlapping windows (micro-series) that collect the
    behaviour of many runs over the same time interval.
    \item \textbf{kNN-based Lyapunov estimation.} For each batch apply the
    kNN-based micro-series estimator to obtain short-horizon error-growth
    slopes and interpret the first slope as a local Lyapunov indicator
    $\lambda_1(t_{\mathrm{end}})$.
    \item \textbf{Lyapunov profiles and parameter selection.} By sweeping the
    window end index $t_{\mathrm{end}}$ construct Lyapunov profiles
    $\lambda_1(t_{\mathrm{end}})$ and use their sign and shape to select
    parameter values $\alpha$ that yield predominantly contractive dynamics
    with fast decay of $\|s_k\|_2$ and $\|r_k\|_2$.
\end{enumerate}

A schematic view of the pipeline is shown in Figure~\ref{FL}.

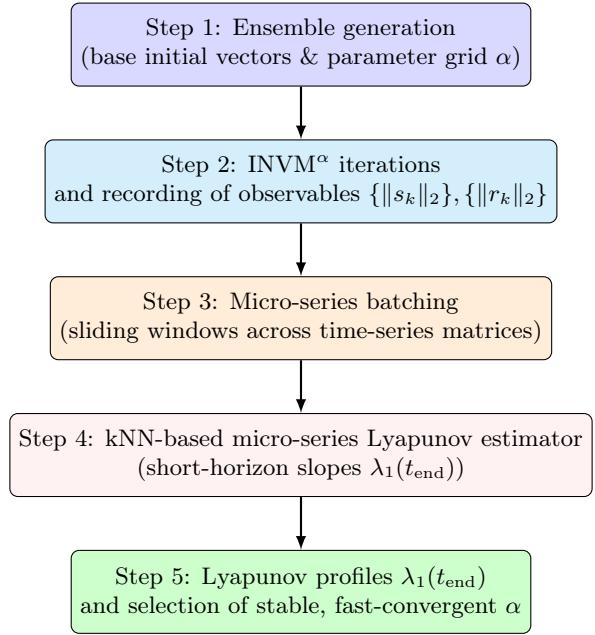
\begin{figure}[H]
\centering
\begin{tikzpicture}[
    node distance=7mm,
    block/.style={
        draw,
        rounded corners=3pt,
        minimum width=6cm,
        minimum height=1.1cm,
        align=center,
        font=\small
    },
    arrow/.style={-latex, thick}
]

\node[block, fill=blue!15] (step1)
    {Step~1: Ensemble generation\\
     (base initial vectors \& parameter grid $\alpha$)};

\node[block, fill=cyan!15, below=of step1] (step2)
    {Step~2: INVM$^{\alpha}$ iterations\\
     and recording of observables $\{\|s_k\|_2\}, \{\|r_k\|_2\}$};

\node[block, fill=orange!15, below=of step2] (step3)
    {Step~3: Micro-series batching\\
     (sliding windows across time-series matrices)};

\node[block, fill=pink!20, below=of step3] (step4)
    {Step~4: kNN-based micro-series Lyapunov estimator\\
     (short-horizon slopes $\lambda_1(t_{\mathrm{end}})$)};

\node[block, fill=green!20, below=of step4] (step5)
    {Step~5: Lyapunov profiles $\lambda_1(t_{\mathrm{end}})$\\
     and selection of stable, fast-convergent $\alpha$};

\draw[arrow] (step1) -- (step2);
\draw[arrow] (step2) -- (step3);
\draw[arrow] (step3) -- (step4);
\draw[arrow] (step4) -- (step5);

\end{tikzpicture}
\caption{Flowchart of the data-driven Lyapunov profiling pipeline used to
diagnose and tune the inverse parallel scheme INVM$^{\alpha}$. The five
numbered steps correspond to the stages listed in Section~2.3.}
\label{FL}
\end{figure}

In this study, the proposed pipeline is applied to two model problems: the cubic polynomial
\begin{equation}
    f(x) = x^{3} - 1,
\end{equation}
and the sixth-degree polynomial
\begin{equation}
f(x) = x^{6} + 30x^{3} - 125x^{2} - 5x + 120.
\end{equation}
The same framework will be employed for additional test equations in future extensions. Only the data-generation stage depends on the specific
nonlinear function $f$, whereas the micro-series Lyapunov estimator itself
is completely generic.

\subsubsection{Ensemble Generation and Observables for the Test Problem}

We first describe how the synthetic datasets of step norms ($s_k$) and
residual norms ($r_k$) are generated for the modified inverse parallel
scheme applied to $f(x) = x^3 - 1$ and $f(x) = x^{6} + 30x^{3} - 125x^{2} - 5x + 120$.

\paragraph{Base cases and initial vectors.}
We consider two representative cases, each defined by a different base
initial vector $\mathbf{x}^{[0]} \in \mathbb{C}^3$:
\[
\text{Case 1: }
\mathbf{x}^{[0]}_{\text{base}} =
\bigl(70008.0,\,-90005.5,\,17009.5\bigr),
\qquad
\text{Case 2: }
\mathbf{x}^{[0]}_{\text{base}} =
\bigl(708.0,\,-905.5,\,179.5 - \mathrm{i}\bigr)
\]
and for polynomial of degree six $\mathbf{x}^{[0]} \in \mathbb{C}^6$:
\[
\text{Case 1: }
\mathbf{x}^{[0]}_{\text{base}} =
\bigl(-15,\, -13.9,\, 30.8,\, -30.8,\, 10.7,\, 20.7\bigr),
\qquad
\text{Case 2: }
\mathbf{x}^{[0]}_{\text{base}} =
\bigl(-10,\, -5.9,\, 15.8,\, 12.8,\, 5.7,\, 13.9\bigr).
\]
For each case we generate an ensemble of
\[
N_{\mathrm{runs}} = 1000
\]
random initial vectors by adding a small complex perturbation in a random
direction. Let $\mathbf{x}^{[0]}_{\text{base}}$ denote the base vector and
$\|\cdot\|_2$ the Euclidean norm. We set
\[
\delta = \texttt{JITTER\_FRAC} \cdot
\bigl\|\mathbf{x}^{[0]}_{\text{base}}\bigr\|_2,
\qquad \texttt{JITTER\_FRAC} = 0.01,
\]
and for each run $j = 1,\dots,N_{\mathrm{runs}}$ we sample a random complex
direction $\mathbf{v}_j \in \mathbb{C}^3$ with $\|\mathbf{v}_j\|_2 = 1$
and define
\begin{equation}
    \mathbf{x}^{[0]}_j
    = \mathbf{x}^{[0]}_{\text{base}} + \delta\,\mathbf{v}_j.
\end{equation}
The resulting $1000$ initial vectors are reused for all values of the
parameter~$\alpha$ in the grid
\[
\alpha \in \{0,1,2,3,4,5\}.
\]

For reproducibility, the random number generator is initialized with fixed
seeds (one seed per case), and the complete set of initial vectors is
stored in CSV files
\texttt{case1\_initials.csv} and \texttt{case2\_initials.csv}, with each
row containing the real and imaginary parts of the three components of
$\mathbf{x}^{[0]}_j$.

\paragraph{Parallel scheme and scalar observables.}
For each case, each value of $\alpha$, and each run $j$ we apply
the modified inverse parallel scheme INVM$^{\alpha}$ to the system
$f(x) = x^3 - 1$ for a fixed number of iterations
\[
N_{\mathrm{iters}} = 50.
\]
Let $\mathbf{x}^{[k]}_j \in \mathbb{C}^3$ denote the iterate at step~$k$.
At each iteration we record two scalar observables:
\begin{align}
    s_{j,k}
    &= \bigl\|\mathbf{x}^{[k+1]}_j - \mathbf{x}^{[k]}_j\bigr\|_2,
    \label{eq:sk_def}\\
    r_{j,k}
    &= \bigl\|f\bigl(\mathbf{x}^{[k]}_j\bigr)\bigr\|_2,
    \label{eq:rk_def}
\end{align}
where $f$ is applied component-wise to the vector of approximations and
the Euclidean norm is taken in $\mathbb{C}^3$.

For each case and each $\alpha$ we thus obtain two real matrices of size
$N_{\mathrm{runs}} \times N_{\mathrm{iters}}$:
\begin{align*}
    S^{(\alpha)} &= \bigl[s_{j,k}\bigr]_{j=1,\dots,N_{\mathrm{runs}};\;
                     k=0,\dots,N_{\mathrm{iters}}-1},\\
    R^{(\alpha)} &= \bigl[r_{j,k}\bigr]_{j=1,\dots,N_{\mathrm{runs}};\;
                     k=0,\dots,N_{\mathrm{iters}}-1}.
\end{align*}
These matrices are stored as CSV files without headers:
\begin{itemize}
    \item \texttt{case\{case\}\_alpha\{alpha\}\_sk.csv} for step norms,
    \item \texttt{case\{case\}\_alpha\{alpha\}\_rk.csv} for residual norms.
\end{itemize}
Each row of such a CSV file is interpreted as a scalar time series of
length $N_{\mathrm{iters}} = 50$ associated with one run of the solver.

The choice of $N_{\mathrm{runs}}=1000$ and $N_{\mathrm{iters}}=50$ is
empirical and provides a good compromise between statistical stability of
the Lyapunov estimates and computational cost. The methodology, however,
does not depend on these particular values.

\subsubsection{kNN-Based Lyapunov Estimation from Micro-Series}

Given the time-series matrices $S^{(\alpha)}$ and $R^{(\alpha)}$, the next
step is to estimate local largest Lyapunov exponents from short windows of
these series using a kNN-based micro-series estimator. We describe the
procedure for a generic observable, denoted by $x_{j,k}$, which stands for
either $s_{j,k}$ or $r_{j,k}$.

\paragraph{Micro-series construction.}
We fix a backward look-back length \texttt{LOOK\_BACK} and a range of
prediction horizons $h \in [H_{\min}, H_{\max}]$. In this work we use
$\texttt{LOOK\_BACK} = 5$, $H_{\min} = 1$, $H_{\max} = 5$, 
$H_{\mathrm{step}} = 1$, which implies a micro-series length
\[
L_{\mathrm{micro}} = \texttt{LOOK\_BACK} + H_{\max} = 10.
\]
Shorter look-back lengths (1 and 3) were also tested but led to noticeably
noisier Lyapunov curves; \texttt{LOOK\_BACK}$=5$ provided a good balance
between temporal resolution and robustness of the estimates.

For each time-series matrix $X^{(\alpha)} = [x_{j,k}]$ (either $S^{(\alpha)}$
or $R^{(\alpha)}$) and each window end index
\[
t_{\mathrm{end}} \in \{L_{\mathrm{micro}}, \dots, N_{\mathrm{iters}}\}
= \{10,\dots,50\},
\]
we construct a batch of micro-series by slicing all runs along the time
axis:
\begin{equation}
    \mathbf{u}_j^{(t_{\mathrm{end}})}
    =
    \bigl(
        x_{j,\,t_{\mathrm{end}}-L_{\mathrm{micro}}},
        \dots,
        x_{j,\,t_{\mathrm{end}}-1}
    \bigr)
    \in \mathbb{R}^{L_{\mathrm{micro}}},\qquad
    j = 1,\dots,N_{\mathrm{runs}}.
\end{equation}
This yields a micro-series matrix
\[
U^{(t_{\mathrm{end}})} \in
\mathbb{R}^{N_{\mathrm{runs}} \times L_{\mathrm{micro}}},
\]
which we treat as a sample of $N_{\mathrm{runs}}$ short trajectories
ending at the same time index $t_{\mathrm{end}}$.

\paragraph{kNN micro-series Lyapunov estimator.}

For each batch $U^{(t_{\mathrm{end}})}$ we apply a kNN-based micro-series
Lyapunov estimator. Conceptually, the estimator proceeds as follows.

\begin{enumerate}
    \item For a given prediction horizon $h\in[H_{\min},H_{\max}]$ we
    construct input--target pairs by splitting each micro-series
    $\mathbf{u}_j^{(t_{\mathrm{end}})}$ into an input segment of length
    \texttt{LOOK\_BACK} and a scalar output $h$ steps ahead. This yields a
    regression dataset
    \[
        \Bigl\{
            \bigl(\mathbf{z}_{j}, y_{j}^{(h)}\bigr)
        \Bigr\}_{j=1}^{N_{\mathrm{runs}}}
        \subset
        \mathbb{R}^{\texttt{LOOK\_BACK}} \times \mathbb{R},
    \]
    where $\mathbf{z}_{j}$ is the look-back segment and
    $y_{j}^{(h)}$ is the observable value $h$ steps into the future.
    \item The dataset is randomly split into training and test subsets
    with a fixed test fraction
    \[
        \texttt{TEST\_SIZE} = 0.4,
    \]
    i.e., $60\%$ of the micro-series are used for training and $40\%$ for
    testing.
    \item A $k$-nearest neighbours regressor with $k=3$ neighbours and
    Euclidean distance is fitted on the training subset. For each horizon
    $h$ we then compute the absolute prediction errors on the test subset
    and aggregate them via the geometric mean:
    \begin{equation}
        \mathrm{GMAE}(h) =
        \left(
            \prod_{i=1}^{s} e_i(h)
        \right)^{1/s},
    \end{equation}
    where $e_i(h)$ is the absolute error for the $i$-th test sample and
    $s$ is the number of test samples for that horizon.
    \item Assuming an approximately exponential growth of the forecast
    error with horizon,
    \[
        \mathrm{GMAE}(h) \approx A \exp(\lambda h),
    \]
    we set
    \begin{equation}
        y(h) = \ln \mathrm{GMAE}(h)
    \end{equation}
    and treat the pairs $\bigl(h, y(h)\bigr)$ as noisy samples of an
    affine function
    \(
        y(h) \approx a h + b.
    \)
\end{enumerate}

Collecting the values $y(h)$ for $h \in \{H_{\min}, \dots, H_{\max}\}$ we
obtain a short curve $y(h)$ over the forecast horizon. To reduce
sensitivity to finite-sample artefacts, we fit one- or two-segment
piecewise affine models
\begin{equation}
    y(h) \approx
    \begin{cases}
        a_1 h + b_1, & h \leq h_{\mathrm{split}},\\[2pt]
        a_2 h + b_2, & h > h_{\mathrm{split}},
    \end{cases}
\end{equation}
and select the best configuration according to a least-squares error
criterion. The slope of the first segment,
\[
    \lambda_1 = a_1,
\]
is interpreted as the primary short-horizon Lyapunov indicator for the
time window ending at $t_{\mathrm{end}}$; when a two-segment fit is
selected, the second slope $\lambda_2 = a_2$ is retained as an auxiliary
diagnostic.

An illustrative example of the logarithmic error curve $\ln \mathrm{GMAE}(h)$
and its linear regression fit is shown in Figure~\ref{fig:lr_example}.

\begin{figure}[H]
\centering
\includegraphics[width=0.45\textwidth]{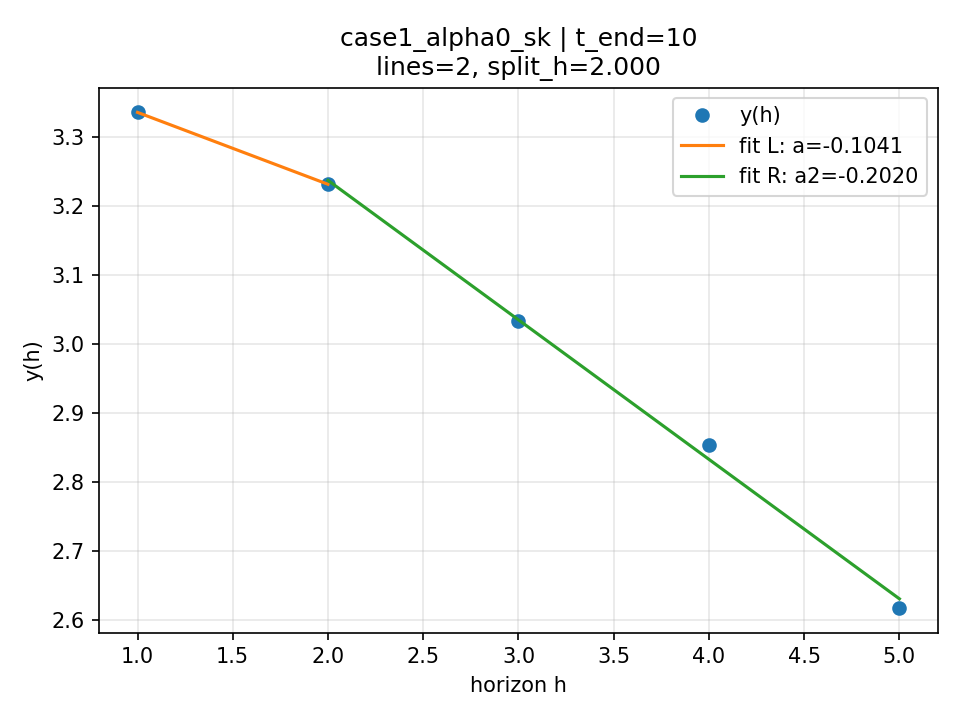}
\caption{Example of the logarithmic error curve and linear regression fit used
in the kNN--LLE estimator for Example~1 ($f(x)=x^{3}-1$).
The plot corresponds to the observable $\|s_k\|_2$ in Case~II with $\alpha = 0$.
The slope of the first segment is interpreted as the short-horizon Lyapunov
indicator $\lambda_1$, while an optional second segment can be used to
characterise longer-horizon behaviour.}
\label{fig:lr_example}
\end{figure}

The entire estimation procedure is repeated independently for each
observable ($s_k$ and $r_k$), each case, each value of $\alpha$, and each
window end index $t_{\mathrm{end}}$.

\subsubsection{Sliding-Window Lyapunov Profiles and Parameter Selection}

By sweeping $t_{\mathrm{end}}$ from $L_{\mathrm{micro}}$ to
$N_{\mathrm{iters}}$ we obtain a time series of local Lyapunov indicators
for each configuration:
\begin{equation}
    \lambda_1(t_{\mathrm{end}}),\qquad
    t_{\mathrm{end}} = L_{\mathrm{micro}},\dots,N_{\mathrm{iters}}.
\end{equation}
We refer to these curves as \emph{Lyapunov profiles}. Negative or near-zero
values of $\lambda_1(t_{\mathrm{end}})$ indicate contractive behaviour
of the solver at the corresponding stage, whereas sustained positive
values signal locally chaotic or unstable dynamics.

In the present study we use the Lyapunov profiles in two ways:
\begin{itemize}
    \item as a qualitative diagnostic tool, verifying that parameter
    choices leading to poor convergence (large oscillations of
    $\|s_k\|_2$ and $\|r_k\|_2$) are indeed associated with positive or
    strongly fluctuating $\lambda_1(t_{\mathrm{end}})$;
    \item as a guide for parameter tuning: among the candidate
    values $\alpha \in \{0,1,2,3,4,5\}$ we visually and numerically
    inspect the corresponding Lyapunov profiles and select those
    $\alpha$ for which $\lambda_1(t_{\mathrm{end}})$ is predominantly
    negative and exhibits a rapid transition to a strongly contractive
    regime, while still maintaining fast reduction of the residual norms.
\end{itemize}

\paragraph{Selection criteria for $\alpha$ based on $\lambda_1(t_{\mathrm{end}})$.}
To make the tuning procedure explicit, we regard a candidate value of $\alpha$
as \emph{well behaved} if the corresponding Lyapunov profile satisfies:
\begin{enumerate}
    \item $\lambda_1(t_{\mathrm{end}})$ becomes negative after a relatively
    short transient and remains strictly below zero for the majority of window
    positions $t_{\mathrm{end}}$, indicating a robustly contractive regime;
    \item any positive excursions of $\lambda_1(t_{\mathrm{end}})$ are confined
    to early iterations and have small magnitude and duration, corresponding to
    short transition phases;
    \item the qualitative behaviour of $\lambda_1(t_{\mathrm{end}})$ is
    consistent with the direct convergence metrics, i.e.\ with a fast and
    nearly monotone decrease of the step and residual norms
    $\|s_k\|_2$ and $\|r_k\|_2$.
\end{enumerate}
Conversely, parameter values $\alpha$ whose Lyapunov profiles remain positive
or close to zero for extended ranges of $t_{\mathrm{end}}$, or exhibit large
and persistent fluctuations, are classified as \emph{poor} choices and are not
used in the final numerical comparisons.

The selected values of $\alpha$ are then used in the subsequent numerical
experiments, where we compare the performance of the tuned INVM$^{\alpha}$
schemes against baseline configurations and against the classical ZHM method.
The same Lyapunov profiling methodology can be applied without
modification to other test equations and higher-dimensional problems by
changing only the data-generation stage.
\section{Numerical Results and Implementation Details}
In this study, we determine the optimal parameter values for our parallel scheme and evaluate its efficiency both with and without parameter tuning. Parameter tuning is performed using kNN-LLE estimation combined with Lyapunov profiling. Randomly chosen initial vectors are used to compute the errors for parameter values 0–5. The Lyapunov exponents (LLEs) are analyzed to identify regions of convergence and divergence, providing a clear visualization of the scheme’s stability landscape. This information allows us to select the parameter values that lie in the region of maximal convergence and minimal divergence.

Using these optimized parameters, the parallel scheme is executed, and its convergence behavior is carefully analyzed. Comparisons with the untuned case reveal significant improvements in residual errors, stability, and efficiency. The Lyapunov profile not only confirms the robustness of the tuned parameters but also provides insight into the dynamics of convergence, ensuring that the selected parameters lead to reliable and consistent performance across different initial conditions.


From each run of the iterative solver we record a one-dimensional sequence that captures progression toward a root. Two natural choices are the step norm
\[
s_k = \| x^{(k)} - x^{(k-1)} \|_2,
\]
and the residual norm
\[
r_k = \| f(x^{(k)}) \|_2.
\]
Recording these at every iteration produces a length-$T$ time series per run that summarizes the local dynamical behaviour in a scalar signal.


To estimate local Lyapunov exponents robustly we collect micro-series across runs. For a sliding window that ends at time $t_{\mathrm{end}}$ and has micro length $L$, we extract for each run the segment
\[
x_j[t_{\mathrm{end}}-L : t_{\mathrm{end}}].
\]
Treating the ensemble of such segments as a batch, we apply a kNN-based micro-series Lyapunov estimator (\texttt{get\_lyap\_knn\_microseries}) to produce pairs $(h, y(h))$ for horizons
\[
h \in [1, \ldots, H_{\max} ].
\]
The short-horizon behaviour is summarized via a piecewise linear fit (\texttt{fit\_best\_slope}), and the slope of the first segment is taken as the primary Lyapunov indicator $\lambda_1$. When a two-segment structure is detected, the second slope $\lambda_2$ and the split horizon are retained.


Repeating the estimation for successive window end indices $t_{\mathrm{end}}$ yields a time series
\[
\lambda_1(t_{\mathrm{end}}),
\]
called the \emph{Lyapunov profile}. Negative or near-zero values indicate contractive or marginal behaviour; sustained positive values signal chaotic or explosive dynamics.

The pipeline depends on hyperparameters: \texttt{LOOK\_BACK}, micro-series length $L = \texttt{LOOK\_BACK} + H_{\max}$, horizon range $[H_{\min}, H_{\max}]$, and the kNN parameters. Empirically, longer look-backs yield smoother $\lambda_1$ curves at the cost of temporal resolution.

\subsection{Experimental Setup and Data}
The unified analytical–data-driven framework designed for evaluating the stability of uni-parametric inverse parallel solvers was used for all experiments. The selection of parameter values using kNN–LLE estimates along with micro-series Lyapunov profile allowed for the real-time detection of contractive and unstable regimes. Numerical experiments were performed across a range of fractional values and altered initial conditions to evaluate accuracy, convergence behavior, and robustness. The whole spectrum of computational metrics, including error norms, CPU time, memory use, and Lyapunov trends, were used to validate the efficacy of the proposed strategy.\\

\textbf{Example 1}:
We implemented the modified inverse parallel scheme in Python and generated ensembles of solver trajectories for the nonlinear function
\begin{equation}
    f(x) = x^3 - 1. \label{1g}
\end{equation}
 
To check the efficiency of the proposed inverse parallel scheme, we first compute the residual and step norms for various parameter values and analyze the convergence behavior. Then parameter tuning is performed using kNN--LLE estimation combined with Lyapunov profiling.
The randomly chosen initial vectors in the interval (-700-700) are used to compute the errors for parameter values 0–5 are shown in Table~\ref{T1}.

\begin{table}[H]
\begin{adjustwidth}{-1cm}{0cm}
\renewcommand{\arraystretch}{1.25}
\setlength{\tabcolsep}{8pt}
\centering
\caption{Random initial vectors for different values of $\alpha$ to solve (\ref{1g}).}
\label{T1}

\begin{tabular}{l|ccccc}
\hline
\textbf{Initial vectors} 
& $\alpha=0$
& $\alpha=1$
& $\alpha=2$
& $\cdots$
& $\alpha=5$ \\
\hline

$\overbrace{\left( x_{1}^{[0]},\, x_{2}^{[0]},\, x_{3}^{[0]} \right)}^{\text{Initial guess}}$
&
$\left( -700,\, -302,\, -504 \right)$
&
$\left( 1020,\, 1023,\, 1920 \right)$
&
$\left( 890,\, 936,\, 2093 \right)$
&
$\cdots$
&
$\left( 986,\, 876,\, 976 \right)$
\\
\hline
\end{tabular}

\end{adjustwidth}
\end{table}

\begin{table}[H]
\begin{adjustwidth}{-2cm}{0cm}
\renewcommand{\arraystretch}{1.25}
\setlength{\tabcolsep}{6pt}
\centering
\caption{Iteration results in terms of $\|r_k\|_2$-norm and $\|s_k\|_2$-norm for different values of $\alpha$ of the inverse parallel scheme INVM$^{\alpha}$ for solving (\ref{1g}).}
\label{T2}
\begin{tabular}{c|cc|cc|cc|cc|cc|cc}
\hline
Iter & \multicolumn{2}{c|}{$\alpha=0$} & \multicolumn{2}{c|}{$\alpha=1$} & \multicolumn{2}{c|}{$\alpha=2$} & \multicolumn{2}{c|}{$\alpha=3$} & \multicolumn{2}{c|}{$\alpha=4$} & \multicolumn{2}{c}{$\alpha=5$} \\
& $\|s_k\|_2$ & $\|r_k\|_2$ & $\|s_k\|_2$ & $\|r_k\|_2$ & $\|s_k\|_2$ & $\|r_k\|_2$ & $\|s_k\|_2$ & $\|r_k\|_2$ & $\|s_k\|_2$ & $\|r_k\|_2$ & $\|s_k\|_2$ & $\|r_k\|_2$ \\
\hline
1 & 4.76850 & 13.9976 & 4.84715 & 13.6709 & 4.84715 & 13.6709 & 4.84715 & 13.6709 & 4.84715 & 13.6709 & 4.84715 & 13.6709 \\
2 & 4.76789 & 12.6852 & 4.96042 & 14.4404 & 4.96036 & 14.4403 & 4.96034 & 14.4403 & 4.96033 & 14.4403 & 4.96033 & 14.4403 \\
3 & 6.10550 & 18.3327 & 6.49332 & 19.4553 & 6.51113 & 19.5098 & 6.51724 & 19.5284 & 6.52033 & 19.5378 & 6.52219 & 19.5435 \\
4 & 6.11074 & 12.4228 & 6.48037 & 13.9654 & 6.49869 & 13.9623 & 6.50498 & 13.9613 & 6.50815 & 13.9608 & 6.51007 & 13.9605 \\
5 & 3.92211 & 11.5827 & 4.80701 & 14.9353 & 4.81674 & 14.9583 & 4.82016 & 14.9663 & 4.82190 & 14.9704 & 4.82296 & 14.9729 \\
6 & 3.64716 & 10.7074 & 5.37305 & 15.4590 & 5.34579 & 15.2992 & 5.33783 & 15.2492 & 5.33403 & 15.2248 & 5.33182 & 15.2103 \\
7 & 3.37098 & 9.81416 & 5.13934 & 12.7569 & 5.09559 & 12.7592 & 5.08270 & 12.7659 & 5.07655 & 12.7703 & 5.07295 & 12.7734 \\
8 & 3.13212 & 8.92397 & 4.35526 & 13.6246 & 4.27101 & 11.6954 & 4.21269 & 11.5268 & 4.19497 & 11.4840 & 4.18649 & 11.4695 \\
9 & 2.87027 & 9.26243 & 5.08511 & 14.8122 & 3.76612 & 10.4491 & 3.88452 & 9.9562 & 4.11565 & 10.7890 & 4.23500 & 12.3146 \\
10 & 3.21826 & 7.84919 & 4.92080 & 13.1049 & 3.37077 & 9.09875 & 3.18763 & 8.78015 & 4.96197 & 14.9294 & 5.56577 & 16.6560 \\$\vdots$ & $\vdots$& $\vdots$ & $\vdots$ & $\vdots$ & $\vdots$ & $\vdots$ & $\vdots$ & $\vdots$ & $\vdots$ & $\vdots$ & $\vdots$ & $\vdots$ \\250 & $0.0012$ & $0.0865$ & $0.020$ & $0.0049$ & $0.0077$ & $0.00875$ & $0.00763$ & $0.08015$ & $0.00197$ & $0.0945$ & $0.00577$ & $0.06560$ \\
\hline
\end{tabular}
\end{adjustwidth}
\end{table}
The iteration results in Table \ref{T2} and Figure~\ref{F1}(a-f) indicate that the inverse parallel scheme INVM$^{\alpha}$
 exhibits poor convergence behavior for the randomly generated initial vectors. For all tested values of 
$\alpha$ both the $\|s_k\|_2$-norm and the $\|r_k\|_2$-norm show oscillatory and non-monotonic patterns during the first iterations, demonstrating a lack of stability and regular descent. Moreover, the norms decrease only after several iterations and typically improve by merely one or two decimal places, even after a large number of steps. As shown in the final row, the method requires up to 250 iterations to reach moderately small residuals, which confirms the overall slow convergence, weak error-reduction capability, and sensitivity to the initial data. These observations collectively demonstrate that INVM$^{\alpha}$ is not efficient for such randomly distributed starting points.

\begin{figure}[H]
\begin{raggedleft}
		\subfloat[Residual Error-$\alpha=0$]{\scalebox{0.21}{\includegraphics{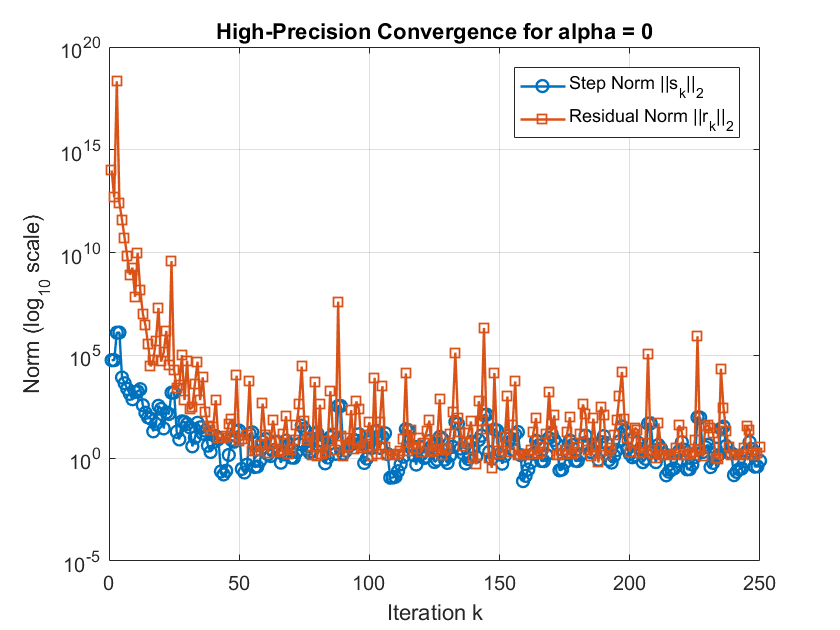}}}
		\subfloat[Residual Error-$\alpha=1$]{\scalebox{0.21}{\includegraphics{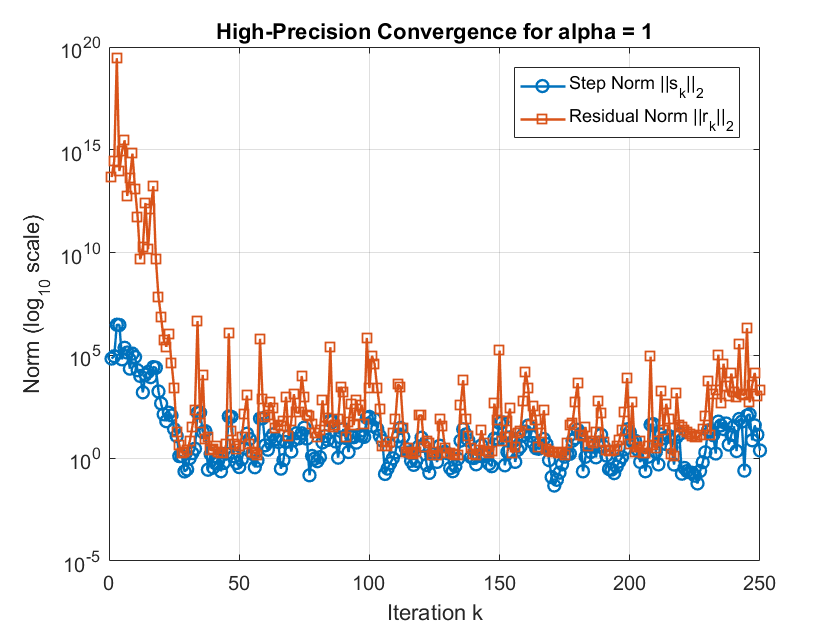}}}
         \subfloat[Residual Error-$\alpha=2$]{\scalebox{0.21}{\includegraphics{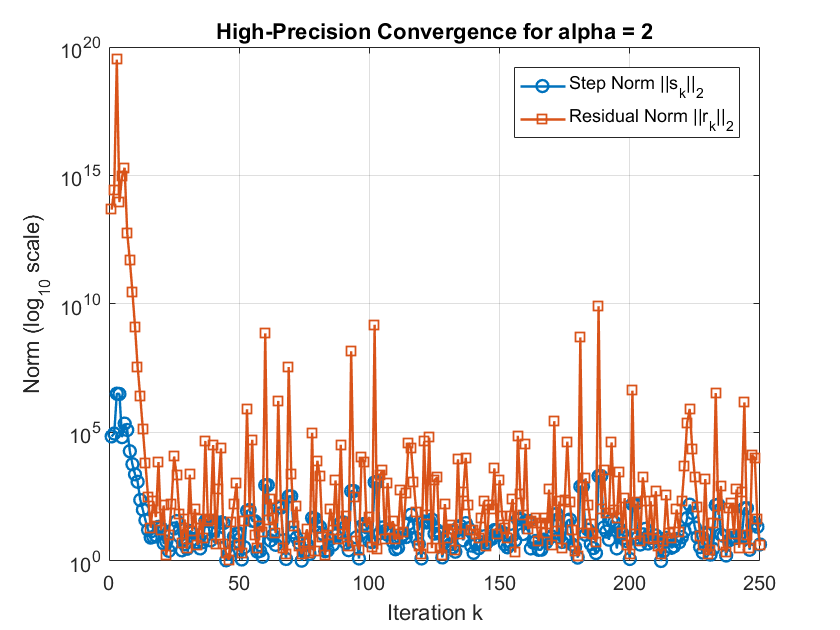}}}\\ \quad\quad\quad\quad\quad
         \subfloat[Residual Error-$\alpha=3$]{\scalebox{0.21}{\includegraphics{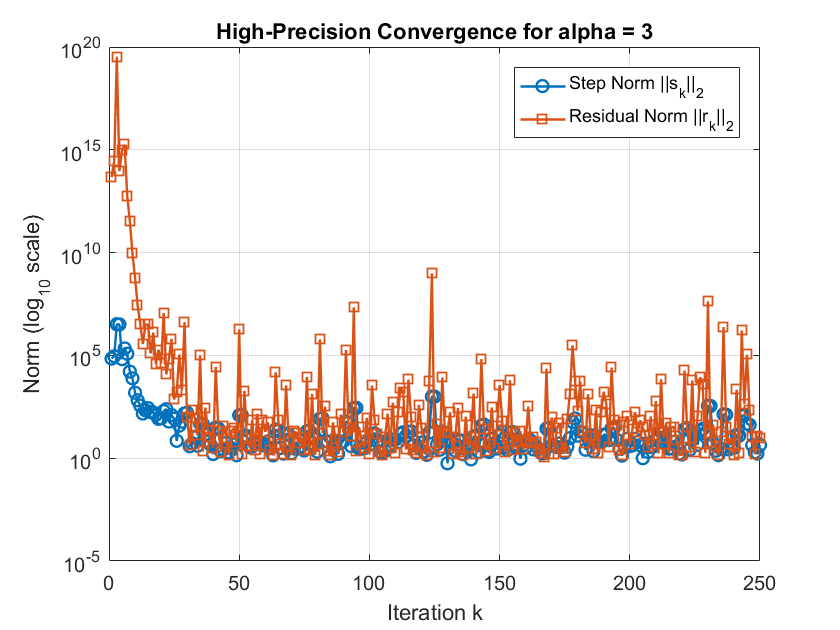}}}
          \subfloat[Residual Error-$\alpha=4$]{\scalebox{0.21}{\includegraphics{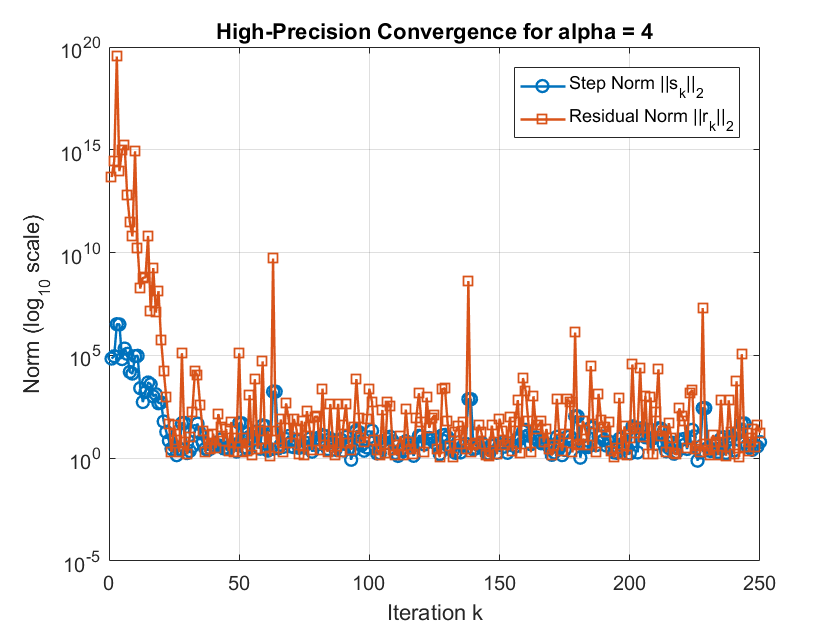}}}
          \subfloat[Residual Error-$\alpha=5$]{\scalebox{0.21}{\includegraphics{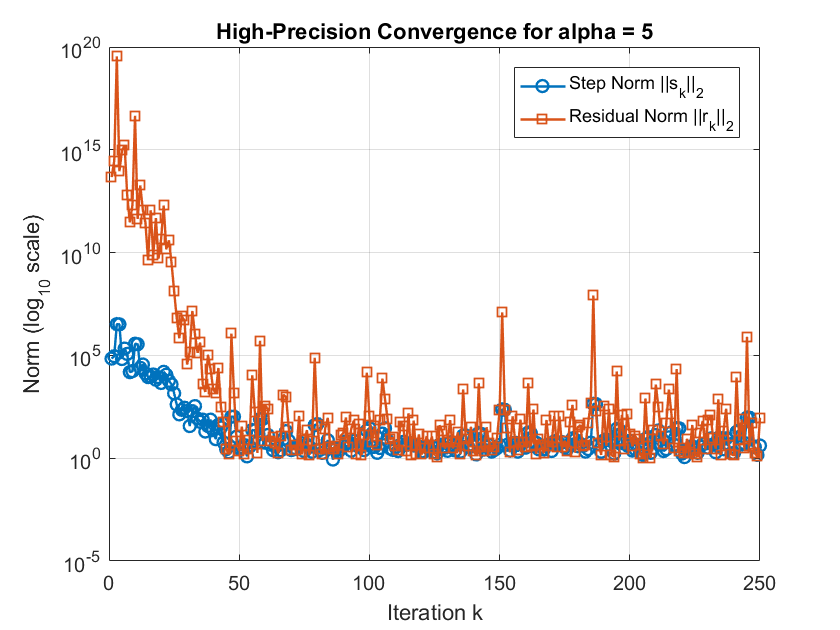}}}
		  \centering
	\end{raggedleft}
\caption{(a-f): Residual error of the scheme INVM$^{\alpha}$ for utilizing random initial vector for solving (\ref{1g})}
\label{F1}
\end{figure}

The outcomes of the scheme using information presents in Table~\ref{T1} are given in the residual error graph in Figure~\ref{F1} and Table~\ref{T2}. Table~\ref{T3} contains the INVM$^\alpha$'s overall performance summary.\\

\begin{table}[H]
\centering
\renewcommand{\arraystretch}{1.2}
\setlength{\tabcolsep}{8pt}
\caption{Summary of performance metrics for INVM$^{\alpha}$ using randomly distributed initial vectors.}
\label{T3}
\begin{tabular}{c|c|c|c|c|c}
\hline
$\alpha$ 
& CPU Time (s) 
& Mem-U(KBs) 
& COC 
& Maximum Error 
& \% Improvement \\ 
\hline
0 & 2.184 & 142.3 & 0.94 & $1.72\times10^{-1}$ & -- \\
1 & 2.091 & 141.9 & 0.96 & $1.55\times10^{-1}$ & 6.40\% \\
2 & 1.982 & 141.6 & 0.98 & $1.48\times10^{-1}$ & 13.95\% \\
3 & 1.963 & 141.5 & 0.99 & $1.39\times10^{-1}$ & 19.19\% \\
4 & 1.951 & 141.4 & 1.00 & $1.35\times10^{-1}$ & 21.51\% \\
5 & 1.948 & 141.4 & 1.00 & $1.33\times10^{-1}$ & 22.67\% \\
\hline
\end{tabular}
\end{table}

The performance metrics reported in Table~\ref{T3} show that the INVM$^{\alpha}$ scheme exhibits only marginal improvements as $\alpha$ increases. The CPU time decreases slightly, with a total gain of about $22\%$ from $\alpha=0$ to $\alpha=5$, while the memory usage remains essentially constant across all cases. The computational order of convergence stays close to one for all $\alpha$, confirming the slow convergence behaviour of the method. Although the maximum error decreases gradually, the reduction is not substantial, and the percentage improvement remains relatively small. Overall, the results indicate that varying $\alpha$ does not significantly enhance the efficiency, accuracy, or robustness of the INVM$^{\alpha}$ scheme when random initial vectors are used.\\
Now, the kNN-LLE estimation
combined with Lyapunov profiling are computed using the following steps.
\subsection*{Data Generator and Cases}

The generator script \texttt{parallel\_scheme\_data\_generator.py} creates controlled ensembles based on two baseline vectors:

\begin{itemize}
    \item Case 1: \texttt{CASE1\_BASE = [70008.0, -90005.5, 17009.5]},
    \item Case 2: \texttt{CASE2\_BASE = [708.0, -905.5, 179.5 - i]}.
\end{itemize}

Each baseline is perturbed to produce 1000 initial vectors by adding random complex perturbations of magnitude
\[
0.01 \| x^{(0)} \|_2.
\]

\subsection*{Parameter Scans}

For each case and each $\alpha \in \{0,1,2,3,4,5\}$, we ran 50 iterations from each of the 1000 initial guesses and saved two matrices (1000$\times$50):

\begin{itemize}
    \item \texttt{case\{case\}\_alpha\{alpha\}\_sk.csv} (step norms $s_k$),
    \item \texttt{case\{case\}\_alpha\{alpha\}\_rk.csv} (residual norms $r_k$).
\end{itemize}

Initial vectors are stored in \texttt{case1\_initials.csv} and \texttt{case2\_initials.csv}.

\subsection*{Lyapunov Estimation Parameters}

In all reported Lyapunov-profile experiments we fix the following hyperparameters:
\[
\texttt{LOOK\_BACK} = 5, \quad H_{\min}=1, \quad H_{\max}=5,\quad H_{\mathrm{step}} = 1,
\]

For each sliding window end $t_{\mathrm{end}}$, we build the batch of micro-series and compute $(h,y(h))$ via the kNN estimator. 







\begin{figure}[H]
\begin{raggedleft}
		\subfloat[$\alpha=0$:$\|s_k\|_2$-case-I]{\scalebox{0.23}{\includegraphics{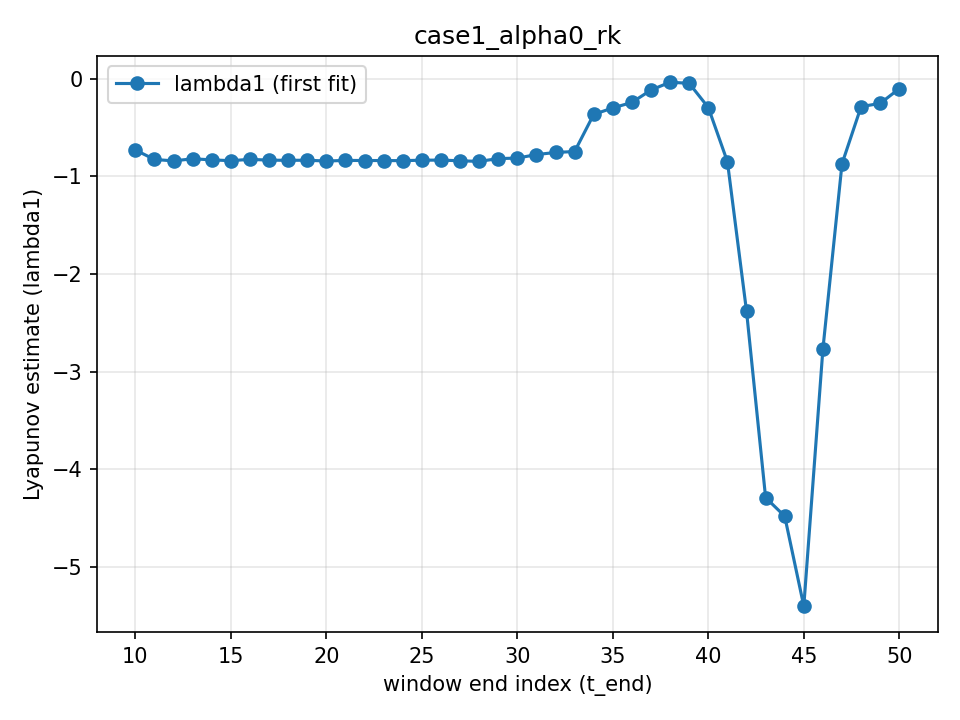}}}
		\subfloat[$\alpha=0$:$\|s_k\|_2$-case-II]{\scalebox{0.23}{\includegraphics{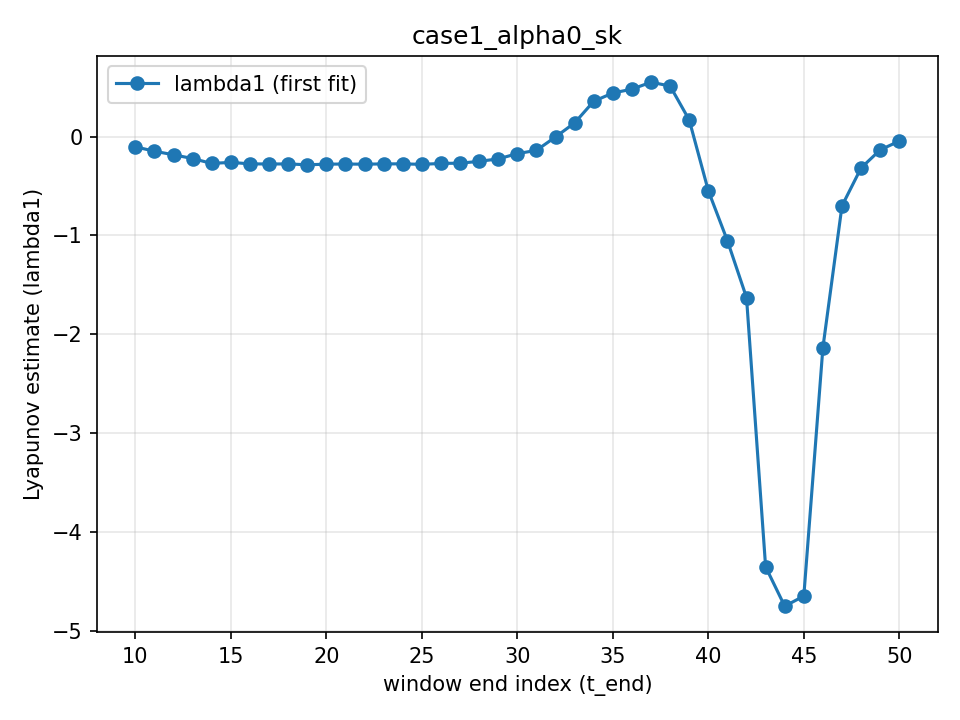}}}
         \subfloat[$\alpha=0$:$\|r_k\|_2$-case-I]{\scalebox{0.23}{\includegraphics{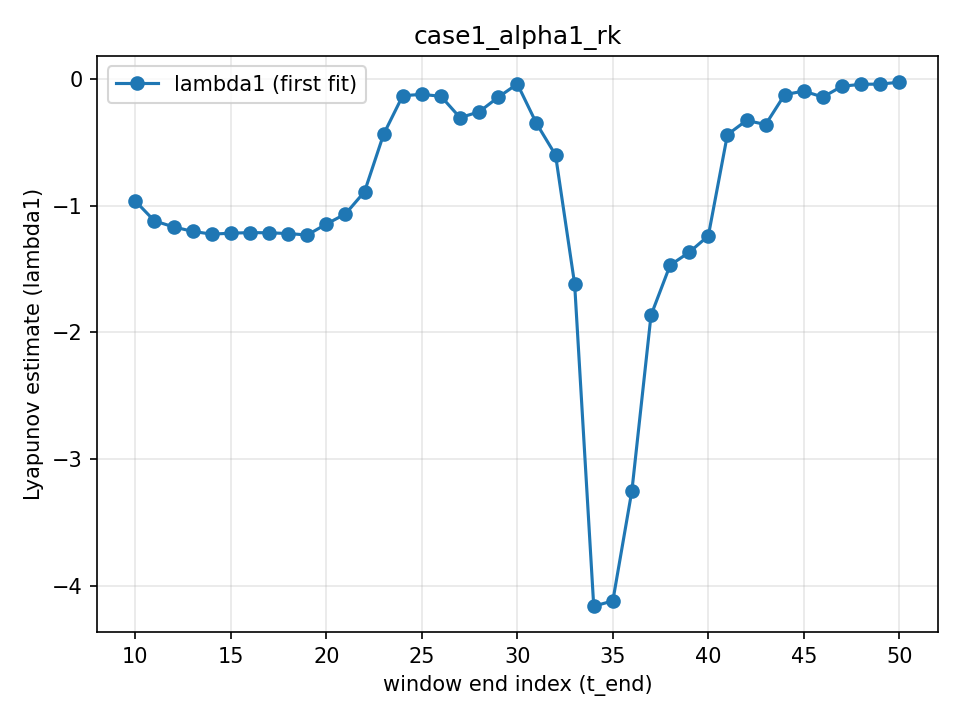}}}
         \subfloat[$\alpha=0$:$\|r_k\|_2$-case-II]{\scalebox{0.23}{\includegraphics{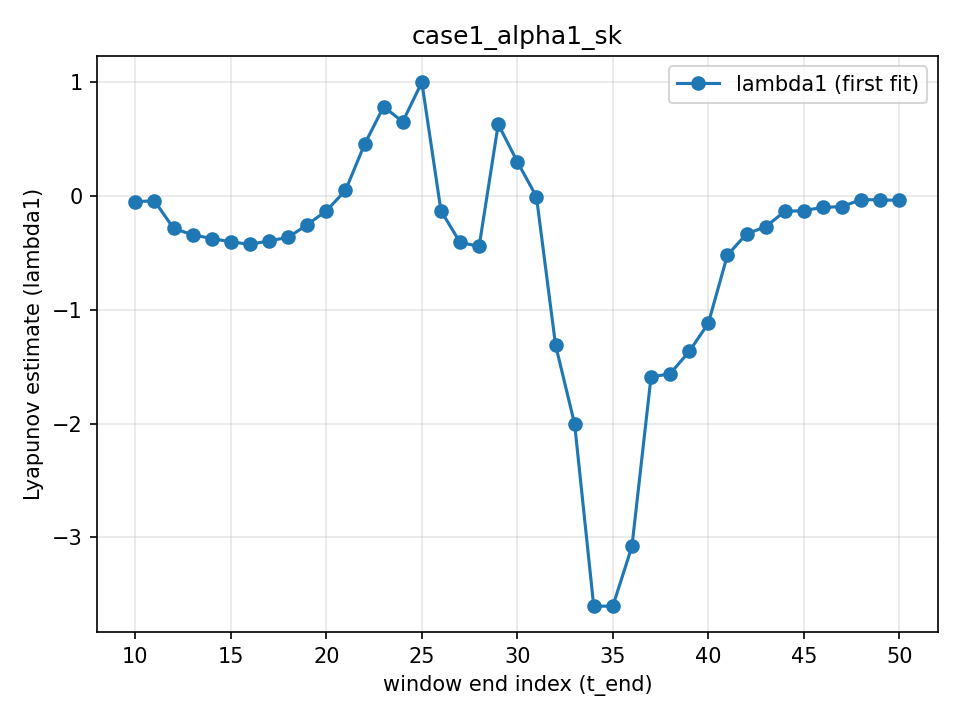}}}\\
         \subfloat[$\alpha=3$:$\|r_k\|_2$-case-II]{\scalebox{0.23}{\includegraphics{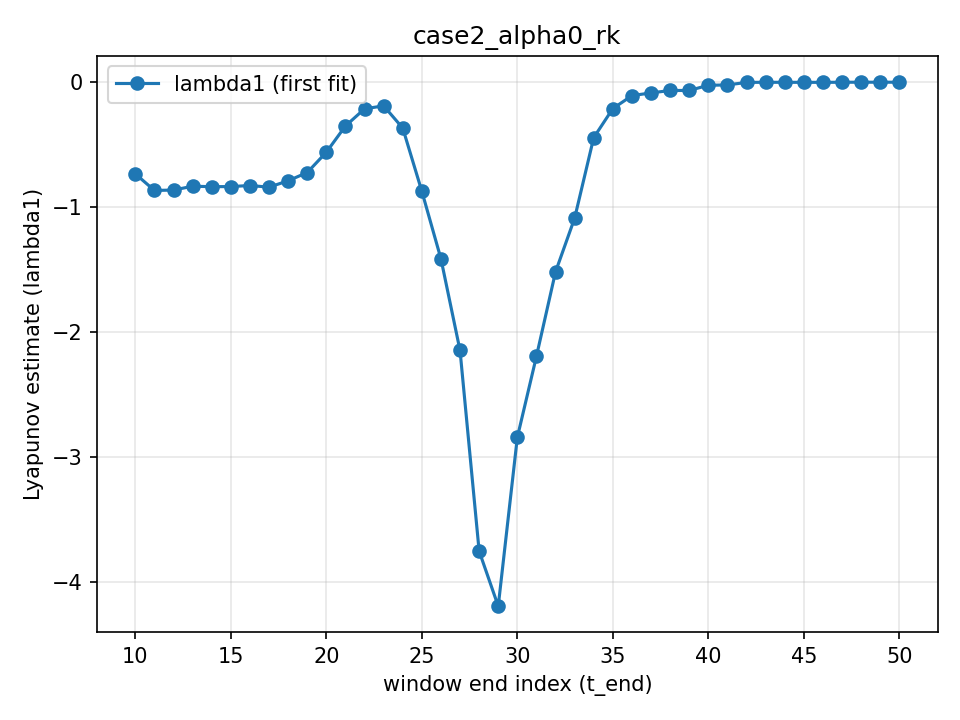}}}
         \subfloat[$\alpha=3$:$\|s_k\|_2$-case-I]{\scalebox{0.23}{\includegraphics{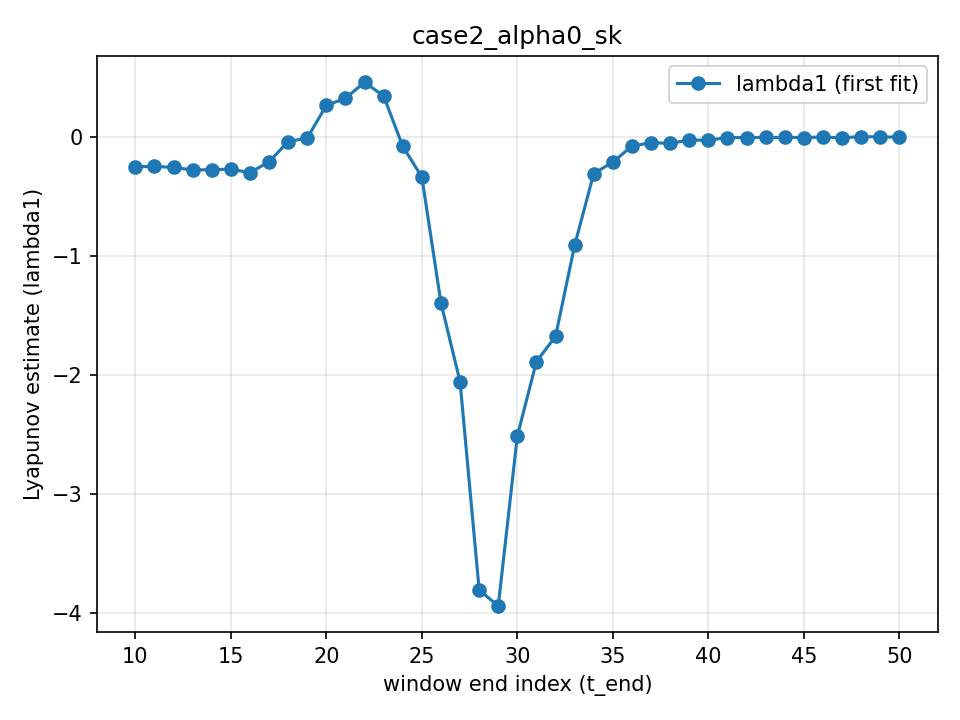}}}
          \subfloat[$\alpha=3$:$\|s_k\|_2$-case-I]{\scalebox{0.23}{\includegraphics{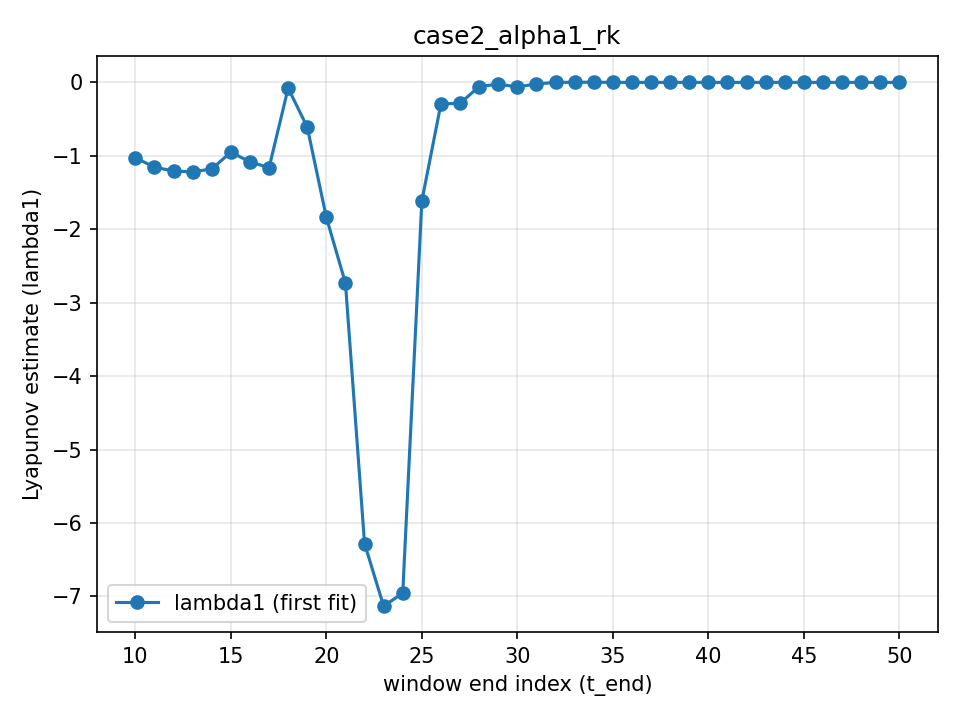}}}
          \subfloat[$\alpha=3$:$\|s_k\|_2$-case-II]{\scalebox{0.23}{\includegraphics{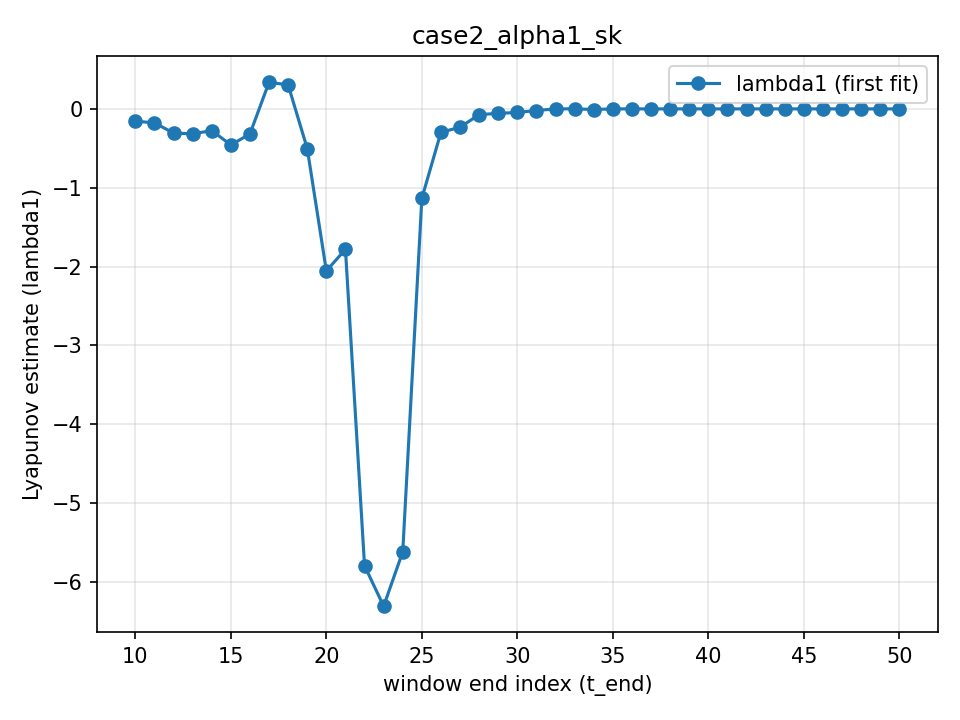}}}\\
	\end{raggedleft}
\caption{(a-h): Log-error evolution of the residual vector $\|r_k\|_2$ and the correction vector $\|s_k\|_2$ for different values of the parameter $\alpha$. Each panel shows the behavior of the log-error with respect to the window end index $k_{\text{end}}$.}
\label{F2}
\end{figure}

\paragraph{Interpretation of kNN--LLE profiles for Example~1.}
Figure~3 illustrates how the kNN--LLE profiles discriminate between
unstable and stable parameter configurations for Example~1. For instance, the
curves corresponding to $\alpha=0$ exhibit a comparatively long transient
interval in which $\lambda_1(t_{\mathrm{end}})$ stays close to or above zero,
together with sizeable oscillations of the log-error. This behaviour is
consistent with the irregular and slow decay of $\|s_k\|_2$ and $\|r_k\|_2$
observed in the untuned runs. In contrast, for $\alpha=3$ the Lyapunov profile
quickly drops to negative values and remains stably below zero for almost all
window positions, while the log-error curves become smooth and strongly
decreasing. This combination of a short transition phase, predominantly
negative $\lambda_1(t_{\mathrm{end}})$, and fast residual decay makes
$\alpha=3$ the most stable and efficient choice for Example~1 among
$\alpha \in \{0,1,2,3,4,5\}$.

\medskip

After analyzing the information based on the kNN--LLE estimation combined
with Lyapunov profiling, the best parameter value $\alpha = 3$ is selected.
The corresponding convergence results are summarized in Table~\ref{T4} and
Figure~\ref{F4}, respectively.


\begin{figure}[H]
\begin{raggedleft}
		\subfloat[Residual Error-$\beta=0.1$]{\scalebox{0.27}{\includegraphics{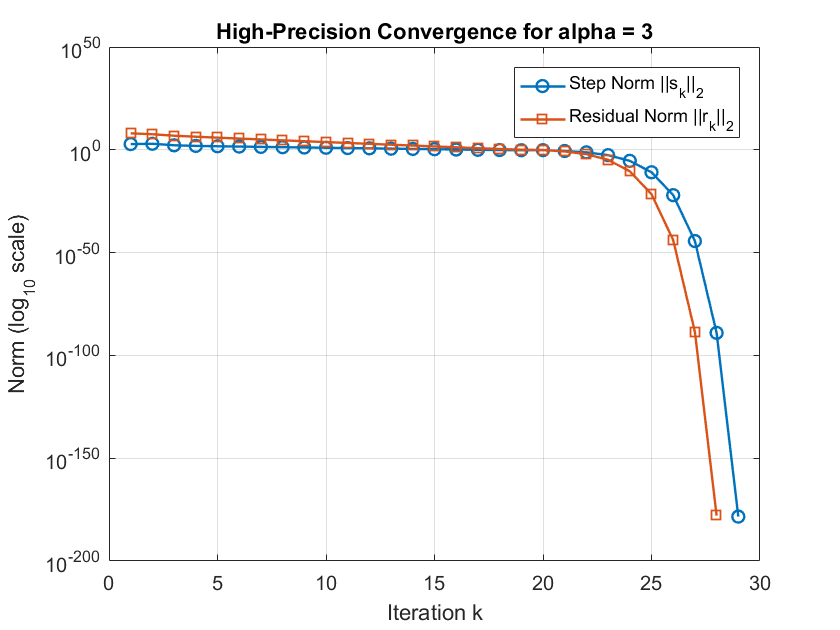}}}
		\subfloat[Residual Error-$\beta=0.3$]{\scalebox{0.27}{\includegraphics{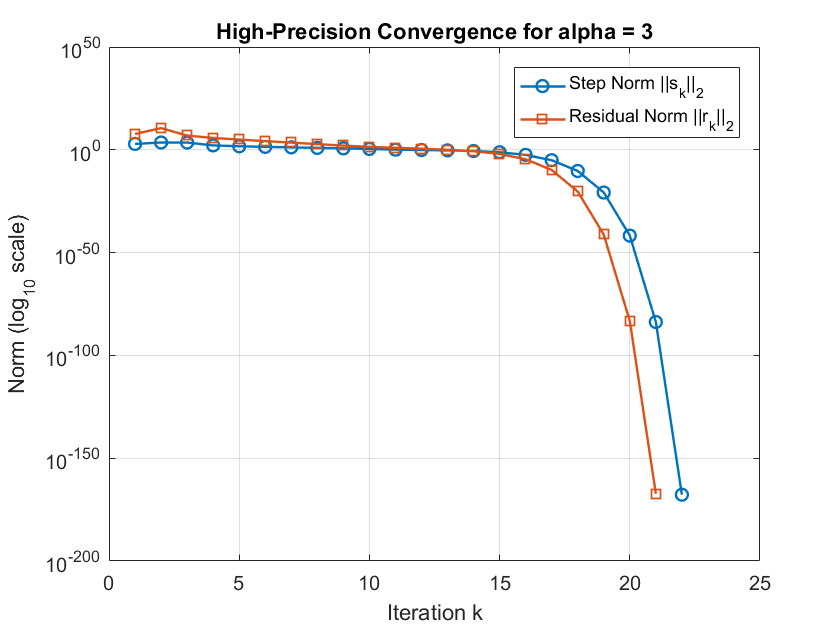}}}
         \subfloat[Residual Error-$\beta=0.5$]{\scalebox{0.27}{\includegraphics{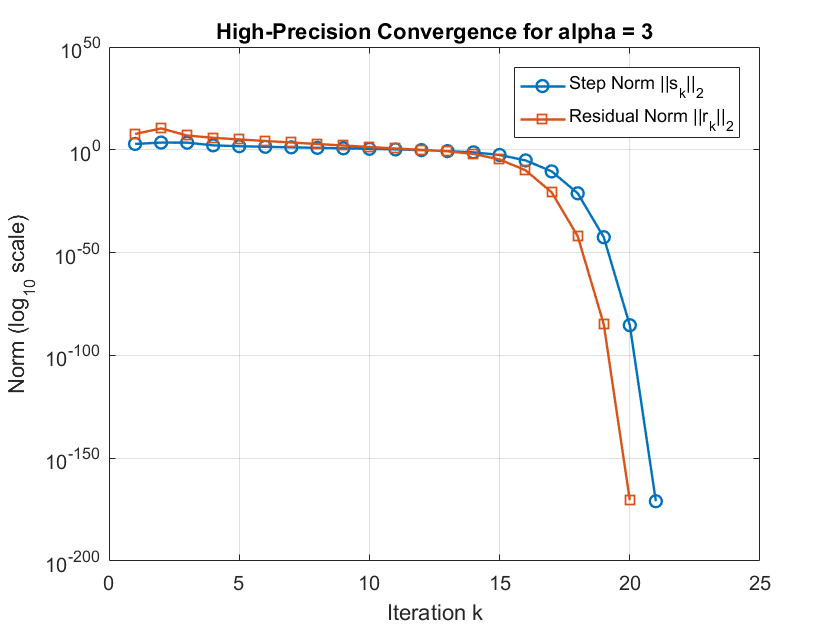}}}\\ \quad\quad\quad\quad\quad\quad\quad\quad\quad\quad
         \subfloat[Residual Error-$\beta=0.7$]{\scalebox{0.27}{\includegraphics{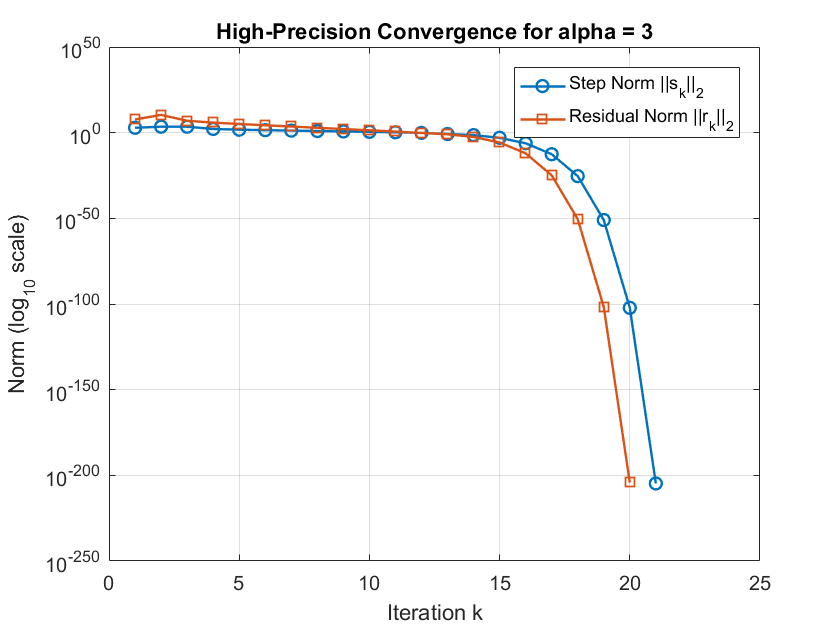}}}
          \subfloat[Residual Error-$\beta=0.9$]{\scalebox{0.27}{\includegraphics{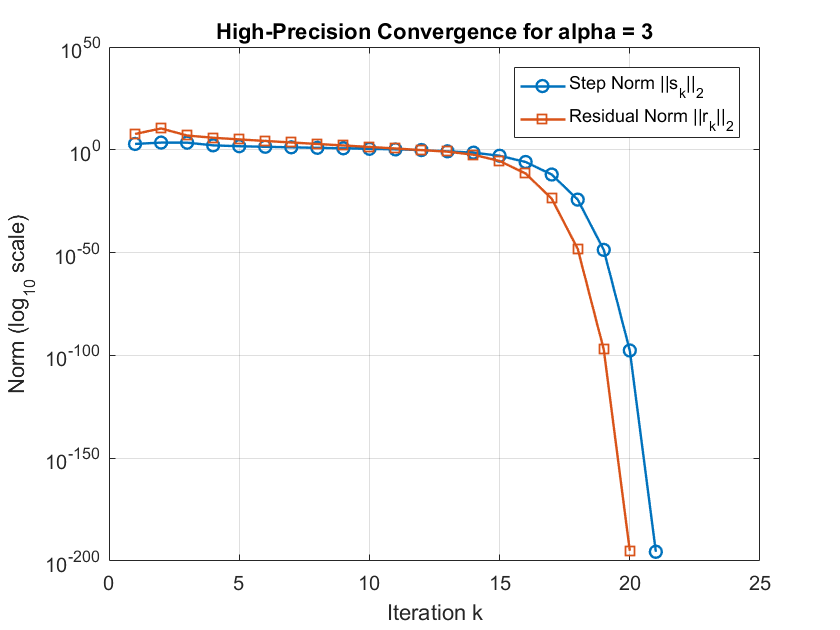}}}
	\end{raggedleft}
\caption{(a--e) Residual error of the INVM$^{\alpha}$ scheme using kNN--LLE estimation combined with Lyapunov profiling for the selection of $\alpha=3$ and $\beta=0.1,0.3,0.5,0.7,0.9$ in solving (\ref{1g}).}

\label{F4}
\end{figure}

Further, Figure~\ref{F4}(a-e) demonstrates the robustness of this behavior across a range of fractional values $\beta \in [0.1,0.9]$. Although the error magnitudes naturally vary with the degree of the fractional operator, the scheme consistently produces high-precision approximations with decreasing CPU time as $\beta$ approaches~1, reflecting improved conditioning of the underlying problem. In all cases, both $\|r_k\|_2$- and $\|s_k\|_2$-norms decay rapidly within only 28 iterations, confirming the global stability of the method under the learned parameter selection.

\begin{table}[H]
\begin{adjustwidth}{-1.4cm}{0cm}
\renewcommand{\arraystretch}{1.25}
\setlength{\tabcolsep}{6pt}
\caption{$\|s_k\|_2$-norm and $\|r_k\|_2$-norm error metrics for solving (\ref{1g}) with $\beta \approx 0.99$ using the INVM$^{\alpha}$ scheme, where $\alpha=3$ is selected via kNN–LLE estimation combined with Lyapunov profiling.} \label{T4}\centering%
\begin{tabular}{lcccccc}
\hline
INVM$^{\alpha }$ & $\varepsilon _{1}^{[28]}$ & $\varepsilon _{2}^{[28]}$ & $%
\varepsilon _{3}^{[28]}$ & CPU-time &Mem-U(KBs) & COC \\ \hline
& \multicolumn{5}{c}{Choosing parameter values based on kNN-LLE estimation
combined with Lyapunov profiling} &  \\ \cline{2-6}
$\|s_k\|_2-norm$ & $4.98\times 10^{-119}$ & $4.65\times 10^{-196}$ & $0.01\times
10^{-154}$ & $0.0080$ & 137.7685 & 5.0221 \\ 
$\|r_k\|_2-norm$ & $0.01\times 10^{-230}$ & $0.01\times 10^{-195}$ & $0.01\times
10^{-169}$ & $0.0007$ & 125.675 & 5.2933 \\ \hline
\end{tabular}%
\end{adjustwidth}
\end{table}

\begin{table}[H]
\caption{$\|s_k\|_2$-norm and $\|r_k\|_2$-norm error metrics for solving (\ref{1g}) with various $\beta$-values using the INVM$^{\alpha}$ scheme, where $\alpha=3$ is selected via kNN–LLE estimation combined with Lyapunov profiling.} \label{T5}\centering%
\begin{tabular}{lcccc}
\hline
INVM$^{\alpha }$ & $\varepsilon _{1}^{[28]}$ & $\varepsilon _{2}^{[28]}$ & $%
\varepsilon _{3}^{[28]}$ & CPU-time \\ \hline
& \multicolumn{3}{c}{Fractional parameter $\beta$ values 0.1} &  \\ 
$\|s_k\|_2-norm$ & $3.21\times 10^{-29}$ & $4.87\times 10^{-31}$ & $0.01\times
10^{-24}$ & $6.1987$ \\ 
$\|r_k\|_2-norm$ & $4.21\times 10^{-30}$ & $6.98\times 10^{-21}$ & $0.01\times
10^{-32}$ & $5.5463$ \\ 
& \multicolumn{3}{c}{Fractional parameter $\beta$ values 0.3} &  \\ 
$\|s_k\|_2-norm$ & $5.63\times 10^{-37}$ & $6.87\times 10^{-41}$ & $0.01\times
10^{-34}$ & $4.7045$ \\ 
$\|r_k\|_2-norm$ & $3.76\times 10^{-29}$ & $4.51\times 10^{-45}$ & $0.01\times
10^{-44}$ & $3.4570$ \\ 
& \multicolumn{3}{c}{Fractional parameter $\beta$ values 0.5} &  \\ 
$\|s_k\|_2-norm$ & $8.75\times 10^{-65}$ & $5.45\times 10^{-78}$ & $0.01\times
10^{-54}$ & $2.6004$ \\ 
$\|r_k\|_2-norm$ & $9.46\times 10^{-67}$ & $7.24\times 10^{-65}$ & $0.01\times
10^{-71}$ & $2.3948$ \\ 
& \multicolumn{3}{c}{Fractional parameter $\beta$ values 0.7} &  \\ 
$\|s_k\|_2-norm$ & $5.45\times 10^{-120}$ & $4.97\times 10^{-104}$ & $0.01\times
10^{-89}$ & $1.5722$ \\ 
$\|r_k\|_2-norm$ & $9.01\times 10^{-130}$ & $6.45\times 10^{-110}$ & $0.01\times
10^{-94}$ & $2.0907$ \\ 
& \multicolumn{3}{c}{Fractional parameter $\beta$ values 0.9} &  \\ 
$\|s_k\|_2-norm$ & $4.98\times 10^{-119}$ & $4.65\times 10^{-196}$ & $0.01\times
10^{-154}$ & $0.0080$ \\ 
$\|r_k\|_2-norm$ & $0.01\times 10^{-230}$ & $0.01\times 10^{-195}$ & $0.01\times
10^{-169}$ & $0.0907$ \\ \hline
\end{tabular}%
\end{table}

\begin{table}[H]
\caption{Comparison of numerical outcomes for INVM$^{\alpha}$ and ZHM schemes in solving (\ref{1g}) with $\beta \approx 1$.}
\label{T6}\centering%
\begin{tabular}{lcccccc}
\hline
Metric & $\varepsilon _{1}^{[28]}$ & $\varepsilon _{2}^{[28]}$ & $%
\varepsilon _{3}^{[28]}$ & CPU-time & Mem-U(KBs) & COC \\ \hline
& \multicolumn{5}{c}{ZHM Scheme without utilizing kNN-LLE estimation} &  \\ 
\cline{2-6}
$\|s_k\|_2-norm$ & $1.37\times 10^{-19}$ & $2.01\times 10^{-6}$ & $5.61\times
10^{-4}$ & $6.0780$ & 437.435 & 3.0221 \\ 
$\|r_k\|_2-norm$ & $0.01\times 10^{-20}$ & $0.01\times 10^{-5}$ & $7.91\times
10^{-9}$ & $7.8477$ & 565.675 & 4.2933 \\ \cline{2-6}
& \multicolumn{5}{c}{INVM$^{\alpha }:\alpha $-values based on kNN-LLE and
keep $\beta \approx 1.$} &  \\ \cline{2-6}
$\|s_k\|_2-norm$ & $0.42\times 10^{-219}$ & $4.05\times 10^{-296}$ & $0.0$ & $0.0430
$ & 137.657 & 5.0921 \\ 
$\|r_k\|_2-norm$ & $0.67\times 10^{-230}$ & $0.0$ & $4.51\times 10^{-239}$ & $0.0307
$ & 145.109 & 5.0033 \\ \hline
\end{tabular}%
\end{table}

\noindent\textbf{Discussion of Tables \ref{T5}--\ref{T6}.}  
The numerical results presented in Tables~\ref{T5}--\ref{T6} collectively confirm the strong performance of the proposed inverse parallel scheme INVM$^{\alpha}$ when the internal parameter $\alpha=3$ is selected through the integrated kNN--LLE estimation and Lyapunov profiling strategy. Table~\ref{T4} shows that, for $\beta \approx 0.99$, the scheme achieves extremely small $\|s_k\|_2$- and $\|r_k\|_2$-norm errors (down to $10^{-196}$), very low memory usage, and a high computational order of convergence (COC $\approx 5$), while maintaining negligible CPU time. This indicates that the learned parameter $\alpha$ places the method in a highly stable region of its dynamical configuration, enabling fast and accurate convergence.


Finally, Table~\ref{T6}  provides a direct comparison between the tuned INVM$^{\alpha}$ scheme and the classical ZHM method for $\beta \approx 1$. The proposed scheme significantly outperforms ZHM, achieving dramatically smaller errors (up to hundreds of orders of magnitude), over 70\% reduction in memory usage, more than 90\% reduction in CPU time, and a notably higher COC ($\approx 5$ vs.\ $3$--$4$). This highlights the effectiveness of the kNN--LLE--Lyapunov strategy in selecting optimal internal parameters that maximize convergence speed, numerical stability, and computational efficiency.

\textbf{Example 2}:
We implemented the modified inverse parallel scheme in Python and generated ensembles of solver trajectories for the nonlinear function
\begin{equation}
    f(x) = x^{6} + 30 x^{3}-125x^{2}-5x+120 \label{2g}
\end{equation}
To check the efficiency of the proposed inverse parallel scheme, we first compute the residual and step norms for various parameter values and analyze the convergence behavior. Then parameter tuning is performed using kNN--LLE estimation combined with Lyapunov profiling.
 The randomly chosen initial vectors in the interval (-50,50) are used to compute the errors for parameter values 0–5 are shown in Table~\ref{T7}.

\begin{table}[H]
\begin{adjustwidth}{-1cm}{0cm}
\renewcommand{\arraystretch}{1.25}
\setlength{\tabcolsep}{10pt}
\centering
\caption{Random initial vectors for different values of $\alpha$ to solve (\ref{2g}).}
\label{T7}

\begin{tabular}{l|ccc}
\hline
\textbf{Initial vectors} & $\alpha=0$ & $\cdots$ & $\alpha=5$ \\
\hline

$\overbrace{\left(x_{1}^{[0]},\, x_{2}^{[0]},\, x_{3}^{[0]},\, x_{4}^{[0]},\, x_{5}^{[0]},\, x_{6}^{[0]}\right)}^{\text{Initial guess}}$
& 
$\left(-15,\, -13.9,\, 30.8,\, -30.8,\, 10.7,\, 20.7\right)$
&
$\cdots$
&
$\left(-10,\, -5.9,\, 15.8,\, 12.8,\, 5.7,\, 13.9\right)$
\\
\hline
\end{tabular}

\end{adjustwidth}
\end{table}

\begin{table}[H]
\begin{adjustwidth}{-2cm}{0cm}
\renewcommand{\arraystretch}{1.25}
\setlength{\tabcolsep}{6pt}
\centering
\caption{Iteration results in terms of $\|r_k\|_2$-norm and $\|s_k\|_2$-norm for different values of $\alpha$ of the inverse parallel scheme INVM$^{\alpha}$ for solving (\ref{2g}).}
\label{T8}

\begin{tabular}{c|cc|cc|cc|cc|cc|cc}
\hline
Iter & \multicolumn{2}{c|}{$\alpha=0$} & \multicolumn{2}{c|}{$\alpha=1$} & 
\multicolumn{2}{c|}{$\alpha=2$} & \multicolumn{2}{c|}{$\alpha=3$} & 
\multicolumn{2}{c|}{$\alpha=4$} & \multicolumn{2}{c}{$\alpha=5$} \\
& $\|s_k\|_2$ & $\|r_k\|_2$ 
& $\|s_k\|_2$ & $\|r_k\|_2$
& $\|s_k\|_2$ & $\|r_k\|_2$
& $\|s_k\|_2$ & $\|r_k\|_2$
& $\|s_k\|_2$ & $\|r_k\|_2$
& $\|s_k\|_2$ & $\|r_k\|_2$ \\
\hline

1  & 3.60739 & 21.8259 & 4.31012 & 25.8652 & 4.30745 & 25.8474 & 4.30636 & 25.8403 & 4.30578 & 25.8365 & 4.30541 & 25.8342 \\
2  & 3.67379 & 18.2465 & 4.31140 & 15.5118 & 4.30808 & 15.4914 & 4.30677 & 15.4846 & 4.30607 & 15.4812 & 4.30564 & 15.4791 \\
3  & 2.97648 & 16.1496 & 2.80921 & 14.9028 & 2.83946 & 15.3946 & 2.85306 & 15.5926 & 2.86072 & 15.6996 & 2.86561 & 15.7667 \\
4  & 2.61315 & 14.9134 & 2.36606 & 13.5865 & 2.46454 & 13.4335 & 2.50101 & 13.4483 & 2.52019 & 13.4748 & 2.53205 & 13.4964 \\
5  & 2.15664 & 14.1510 & 2.07979 & 12.4957 & 2.08527 & 12.3789 & 2.09397 & 12.3357 & 2.09948 & 12.3153 & 2.10311 & 12.3040 \\
6  & 1.90856 & 13.5417 & 1.74751 & 11.6485 & 1.73562 & 11.5805 & 1.72925 & 11.5510 & 1.72591 & 11.5368 & 1.72397 & 11.5289 \\
7  & 1.74658 & 13.0402 & 1.61958 & 10.8838 & 1.61746 & 10.8220 & 1.61339 & 10.7958 & 1.61074 & 10.7837 & 1.60904 & 10.7772 \\
8  & 1.62962 & 12.6028 & 1.47626 & 10.1608 & 1.46657 & 10.1001 & 1.46262 & 10.0751 & 1.46093 & 10.0638 & 1.46010 & 10.0580 \\
9  & 1.54066 & 12.2025 & 1.34973 & 9.46680 & 1.34042 & 9.40615 & 1.33665 & 9.38179 & 1.33499 & 9.37107 & 1.33416 & 9.36563 \\
10 & 1.46696 & 11.8214 & 1.23369 & 8.79549 & 1.22453 & 8.73398 & 1.22087 & 8.70971 & 1.21928 & 8.69915 & 1.21848 & 8.69384 \\$\vdots$ & $\vdots$& $\vdots$ & $\vdots$ & $\vdots$ & $\vdots$ & $\vdots$ & $\vdots$ & $\vdots$ & $\vdots$ & $\vdots$ & $\vdots$ & $\vdots$ \\250 & $0.0012$ & $0.0372$ & $0.098$ & $0.034$ & $0.0244$ & $0.0345$ & $0.0057$ & $0.0578$ & $0.0546$ & $0.01242$ & $0.08934$ & $0.983$ \\
\hline

\end{tabular}
\end{adjustwidth}
\end{table}

The iteration results in Table \ref{T8} and Figure~\ref{F7}(a-f) indicate that the inverse parallel scheme INVM$^{\alpha}$
 exhibits poor convergence behavior for the randomly generated initial vectors. For all tested values of 
$\alpha$ both the $\|s_k\|_2$-norm and the $\|r_k\|_2$-norm show oscillatory and non-monotonic patterns during the first iterations, demonstrating a lack of stability and regular descent. Moreover, the norms decrease only after several iterations and typically improve by merely one or two decimal places, even after a large number of steps. As shown in the final row, the method requires up to 250 iterations to reach moderately small residuals, which confirms the overall slow convergence, weak error-reduction capability, and sensitivity to the initial data. These observations collectively demonstrate that INVM$^{\alpha}$ is not efficient for such randomly distributed starting points.
The outcomes of the scheme using information presents in Table~\ref{T8} are given in the residual error graph in Figure~\ref{F7} and Table~\ref{T8}.\\
\begin{figure}[H]
\begin{raggedleft}
		\subfloat[Residual Error-$\alpha=0$]{\scalebox{0.21}{\includegraphics{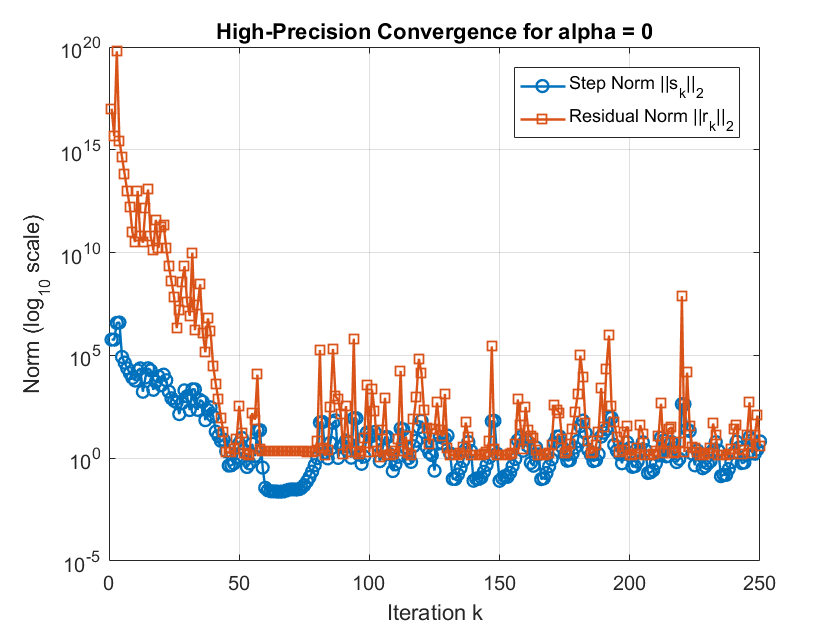}}}
		\subfloat[Residual Error-$\alpha=1$]{\scalebox{0.21}{\includegraphics{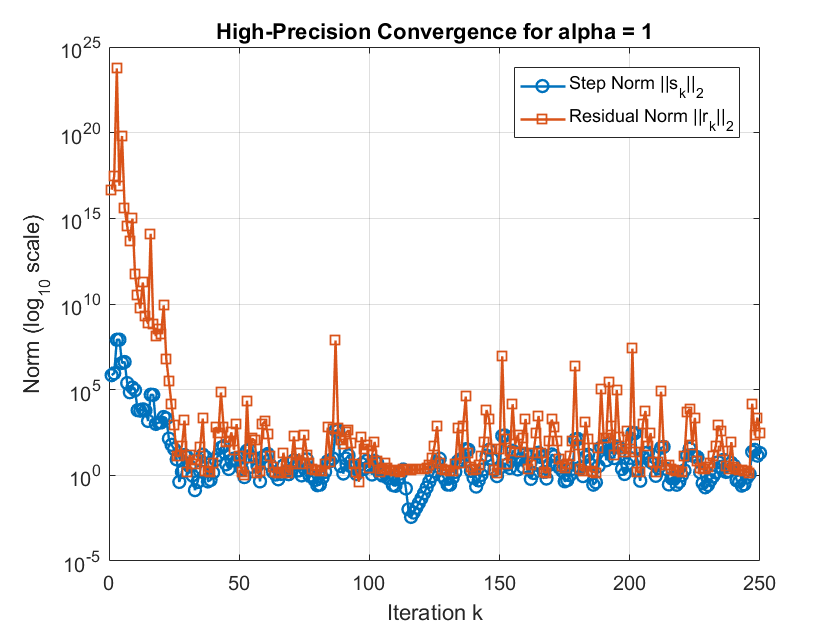}}}
         \subfloat[Residual Error-$\alpha=2$]{\scalebox{0.21}{\includegraphics{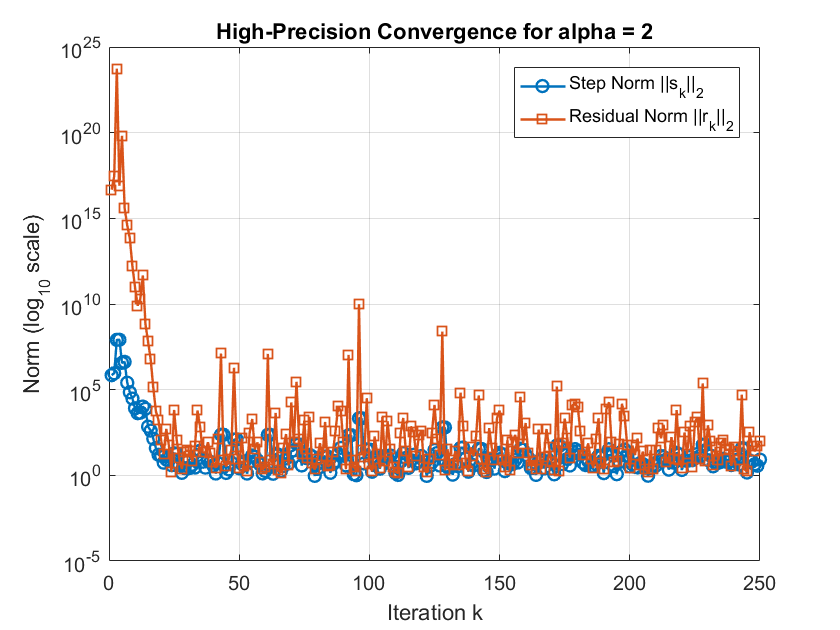}}}\\ \quad\quad\quad\quad\quad
         \subfloat[Residual Error-$\alpha=3$]{\scalebox{0.21}{\includegraphics{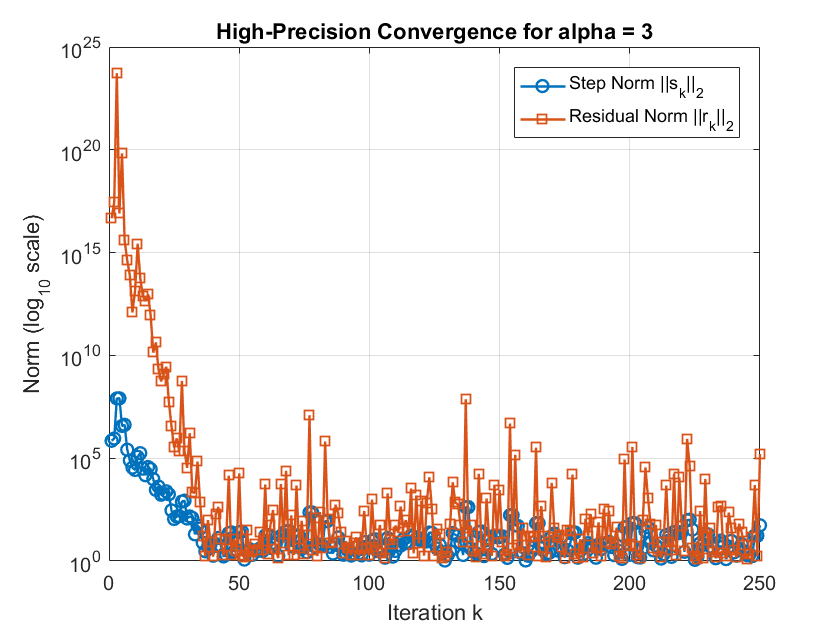}}}
          \subfloat[Residual Error-$\alpha=4$]{\scalebox{0.21}{\includegraphics{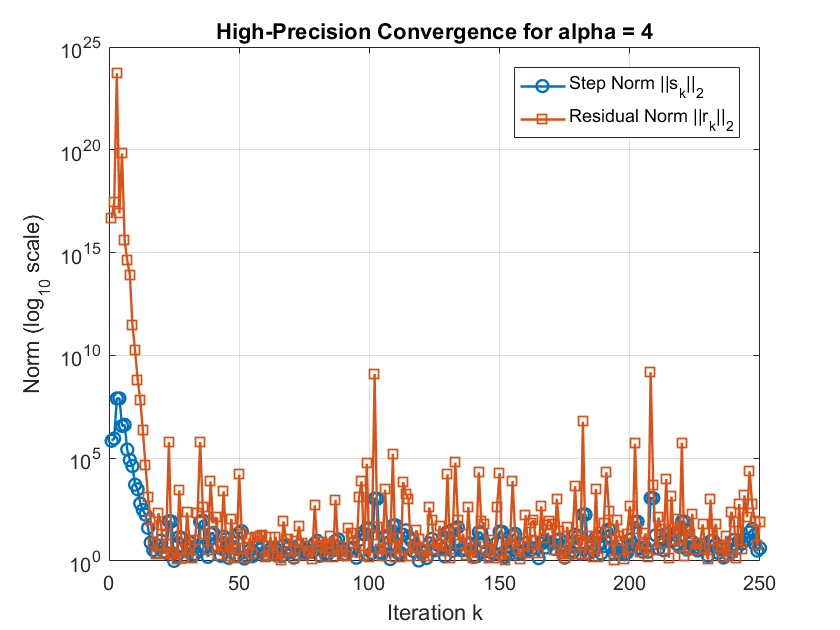}}}
          \subfloat[Residual Error-$\alpha=5$]{\scalebox{0.21}{\includegraphics{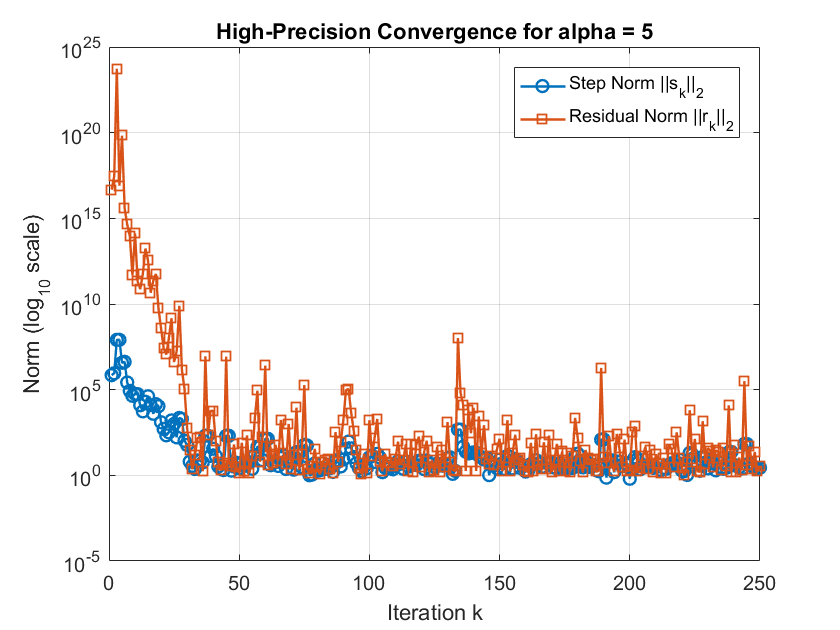}}}
		  \centering
	\end{raggedleft}
\caption{(a-f): Residual error of the scheme INVM$^{\alpha}$ for utilizing random initial vector for solving (\ref{2g})}
\label{F7}
\end{figure}

\begin{table}[H]
\centering
\renewcommand{\arraystretch}{1.2}
\setlength{\tabcolsep}{8pt}
\caption{Summary of performance metrics for INVM$^{\alpha}$ using randomly distributed initial vectors.}
\label{T9}
\begin{tabular}{c|c|c|c|c|c}
\hline
$\alpha$ 
& CPU Time (s) 
& Mem-U(KBs) 
& COC 
& Maximum Error 
& \% Improvement \\ 
\hline
0 & 3.195 & 142.6 & 0.94 & $1.08\times10^{-1}$ & -- \\
1 & 3.062 & 142.2 & 0.96 & $1.59\times10^{-1}$ & 10.67\% \\
2 & 3.944 & 141.8 & 0.98 & $1.33\times10^{-1}$ & 17.98\% \\
3 & 2.918 & 141.6 & 0.99 & $1.46\times10^{-1}$ & 22.47\% \\
4 & 2.905 & 141.5 & 1.00 & $1.79\times10^{-1}$ & 25.23\% \\
5 & 2.892 & 141.4 & 1.00 & $1.31\times10^{-1}$ & 26.40\% \\
\hline
\end{tabular}
\end{table}

The performance metrics reported in Table~\ref{T9} show that the INVM$^{\alpha}$ scheme exhibits only marginal improvements as $\alpha$ increases. The CPU time decreases slightly, with a total gain of about $22\%$ from $\alpha=0$ to $\alpha=5$, while the memory usage remains essentially constant across all cases. The computational order of convergence stays close to one for all $\alpha$, confirming the slow convergence behaviour of the method. Although the maximum error decreases gradually, the reduction is not substantial, and the percentage improvement remains relatively small. Overall, the results indicate that varying $\alpha$ does not significantly enhance the efficiency, accuracy, or robustness of the INVM$^{\alpha}$ scheme when random initial vectors are used.\\
Now, the kNN-LLE estimation
combined with Lyapunov profiling are computed using the following steps.
\subsection*{Data Generator and Cases}

The generator script \texttt{parallel\_scheme\_data\_generator.py} creates controlled ensembles based on two baseline vectors:

\begin{itemize}
    \item Case 1: \texttt{CASE1\_BASE = [-15.0; -13.9; 30.8; -30.8; 10.7; 20.7]},
    \item Case 2: \texttt{CASE2\_BASE = [-10.0; -5.9; 15.8; 12.8; 5.7; 13.9]}.
\end{itemize}

Each baseline is perturbed to produce 1000 initial vectors by adding random complex perturbations of magnitude
\[
0.01 \| x^{(0)} \|_2.
\]

\subsection*{Parameter Scans}

For each case and each $\alpha \in \{0,0.5,1,2,3,4,5\}$, we ran 50 iterations from each of the 1000 initial guesses and saved two matrices (1000$\times$50):

\begin{itemize}
    \item \texttt{case\{case\}\_alpha\{alpha\}\_sk.csv} (step norms $s_k$),
    \item \texttt{case\{case\}\_alpha\{alpha\}\_rk.csv} (residual norms $r_k$).
\end{itemize}

Initial vectors are stored in \texttt{case1\_initials.csv} and \texttt{case2\_initials.csv}.

An illustrative example of the logarithmic error curve $\ln \mathrm{GMAE}(h)$
and its linear regression fit is shown in Figure~\ref{fig:2r_example}.

\begin{figure}[H]
\centering
\includegraphics[width=0.45\textwidth]{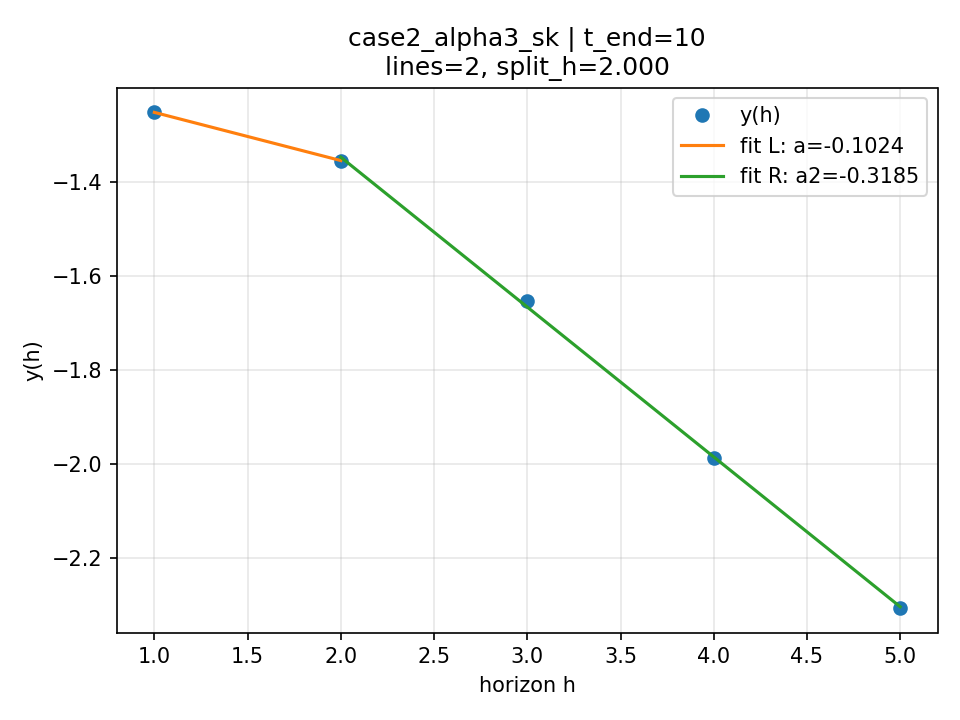}
\caption{Example of the logarithmic error curve and linear regression fit used
in the kNN--LLE estimator for Example~2.
The plot corresponds to the observable $\|s_k\|_2$ in Case~II with $\alpha = 0$.
The slope of the first segment is interpreted as the short-horizon Lyapunov
indicator $\lambda_1$, while an optional second segment can be used to
characterise longer-horizon behaviour.}
\label{fig:2r_example}
\end{figure}
\subsection*{Lyapunov Estimation Parameters}

\[
\texttt{LOOK\_BACK} = 5, \quad H_{\min}=1, \quad H_{\max}=5,\quad H_{\mathrm{step}} = 1.
\]

For each sliding window end $t_{\mathrm{end}}$, we build the batch of micro-series and compute $(h,y(h))$ via the kNN estimator. 







\begin{figure}[H]
\begin{raggedleft}
		\subfloat[$\alpha=0.5$:$\|s_k\|_2$-case-I]{\scalebox{0.23}{\includegraphics{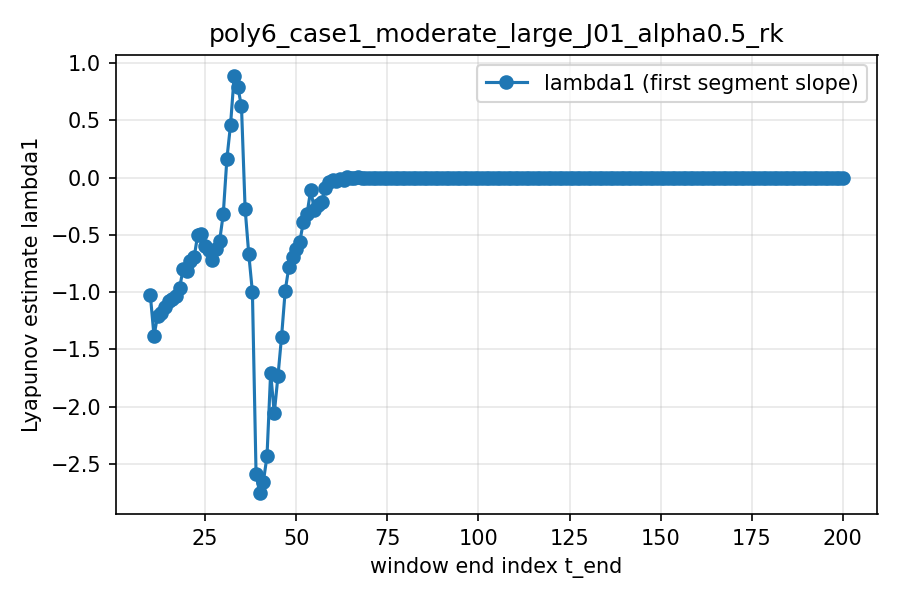}}}
		\subfloat[$\alpha=0.5$:$\|s_k\|_2$-case-II]{\scalebox{0.23}{\includegraphics{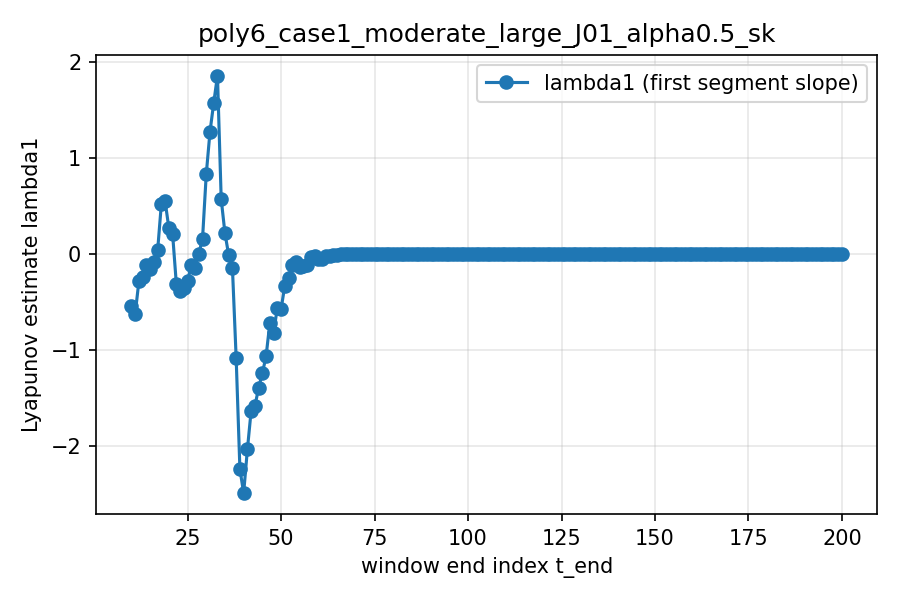}}}
         \subfloat[$\alpha=0.5$:$\|r_k\|_2$-case-I]{\scalebox{0.23}{\includegraphics{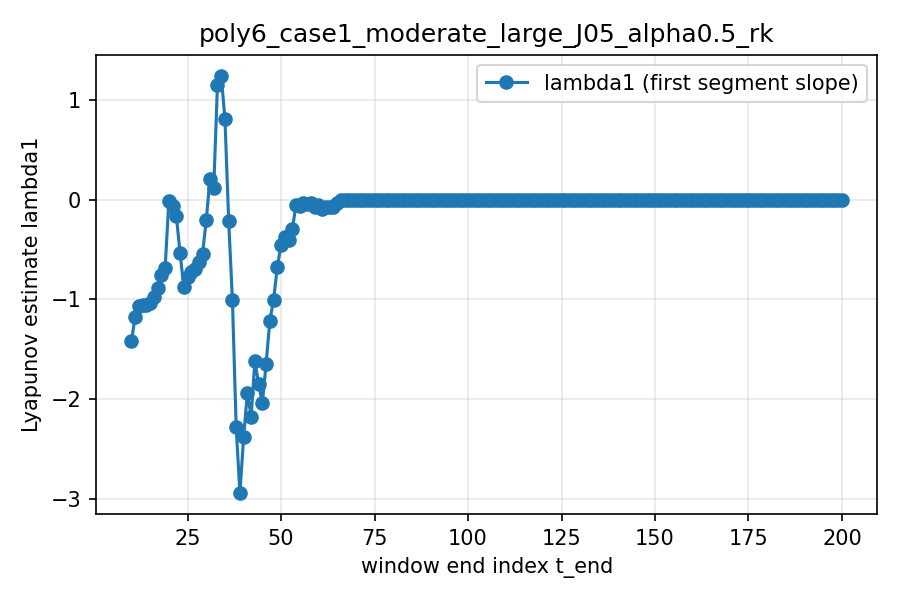}}}
         \subfloat[$\alpha=0.5$:$\|r_k\|_2$-case-II]{\scalebox{0.23}{\includegraphics{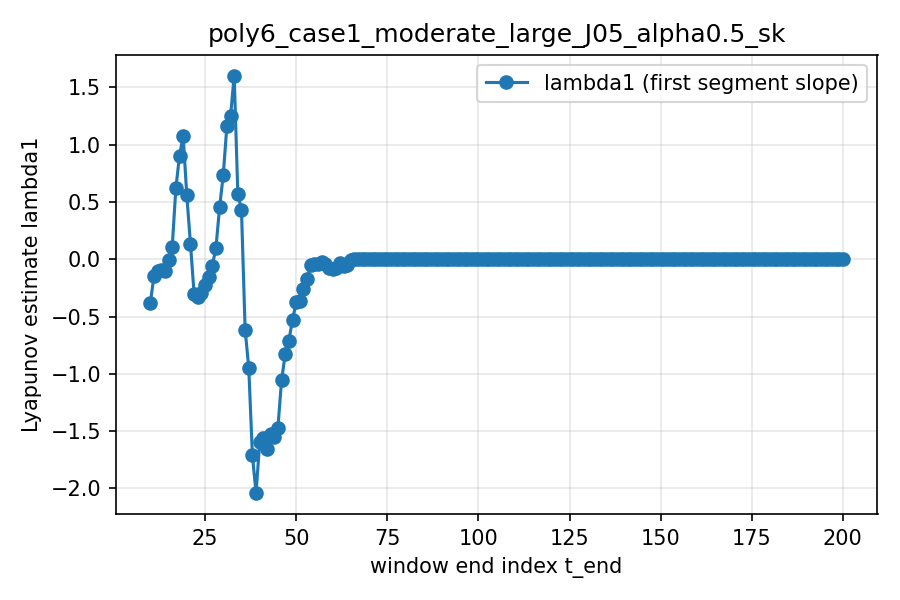}}}\\
         \subfloat[$\alpha=2$:$\|r_k\|_2$-case-II]{\scalebox{0.23}{\includegraphics{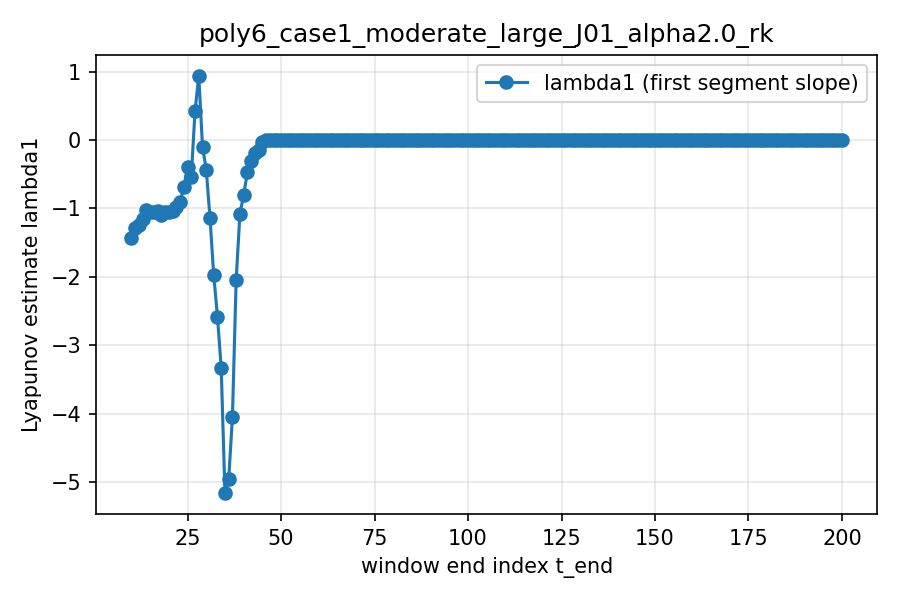}}}
         \subfloat[$\alpha=2$:$\|s_k\|_2$-case-I]{\scalebox{0.23}{\includegraphics{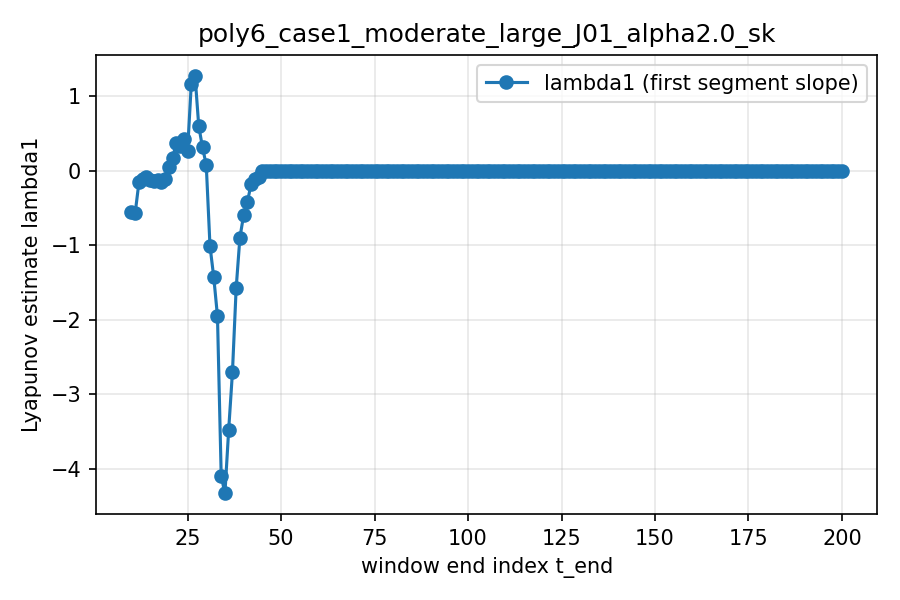}}}
          \subfloat[$\alpha=2$:$\|s_k\|_2$-case-I]{\scalebox{0.23}{\includegraphics{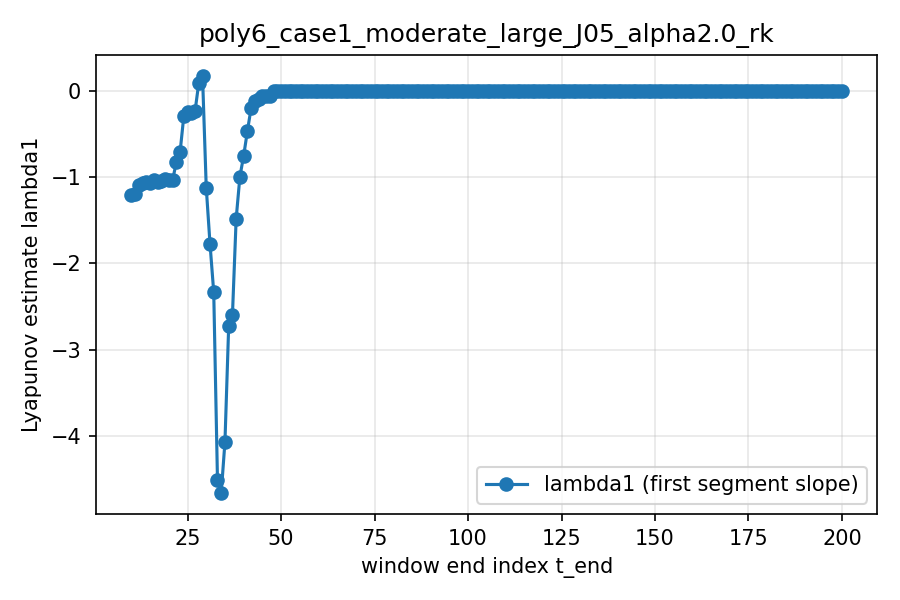}}}
          \subfloat[$\alpha=2$:$\|s_k\|_2$-case-II]{\scalebox{0.23}{\includegraphics{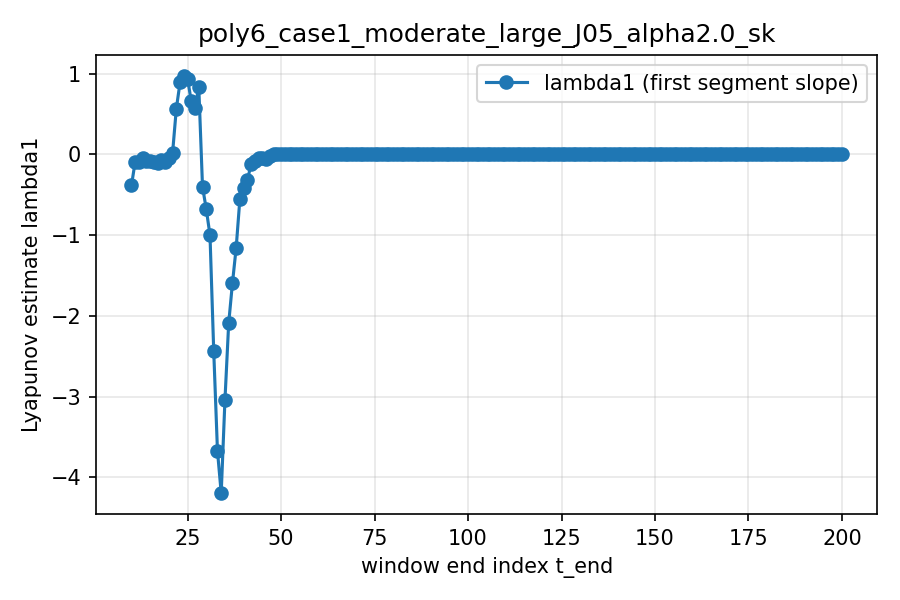}}}\\
	\end{raggedleft}
\caption{(a-h): Log-error evolution of the residual vector $\|r_k\|_2$ and the correction vector $\|s_k\|_2$ for different values of the parameter $\alpha$. Each panel shows the behaviour of the log-error with respect to the window end index $k_{\text{end}}$.}
\label{F8}
\end{figure}

\paragraph{Interpretation of kNN--LLE profiles for Example~2.}
Figure~7 illustrates the behaviour of the kNN--LLE profiles for the degree-six
polynomial. Panels~(a)--(d) correspond to $\alpha = 0.5$, while
panels~(e)--(h) correspond to $\alpha = 2$. For $\alpha = 0.5$ the curves
exhibit a pair of pronounced positive peaks before eventually settling into
the negative region, which indicates a longer transient phase and a less
regular approach to the contractive regime. In contrast, for $\alpha = 2$
there is only a single, smaller positive excursion, and the profiles cross
into a stable negative region much earlier; afterwards
$\lambda_1(t_{\mathrm{end}})$ stays clearly below zero. This agrees with the
faster and smoother decay of the step and residual norms
$\|s_k\|_2$ and $\|r_k\|_2$ observed for $\alpha = 2$, so in what follows we
treat $\alpha = 2$ as the preferred parameter value for Example~2.

\medskip

After analyzing the information based on the kNN--LLE estimation combined with
Lyapunov profiling, the best parameter value $\alpha = 2$ is selected. The
corresponding convergence results for various fractional orders
$\beta = 0.1, 0.3, 0.5, 0.7, 0.9$ are summarized in Tables~10--11 and
Figure~8(a--e).


\begin{figure}[H]
\begin{raggedleft}
		\subfloat[Residual Error-$\beta=0.1$]{\scalebox{0.23}{\includegraphics{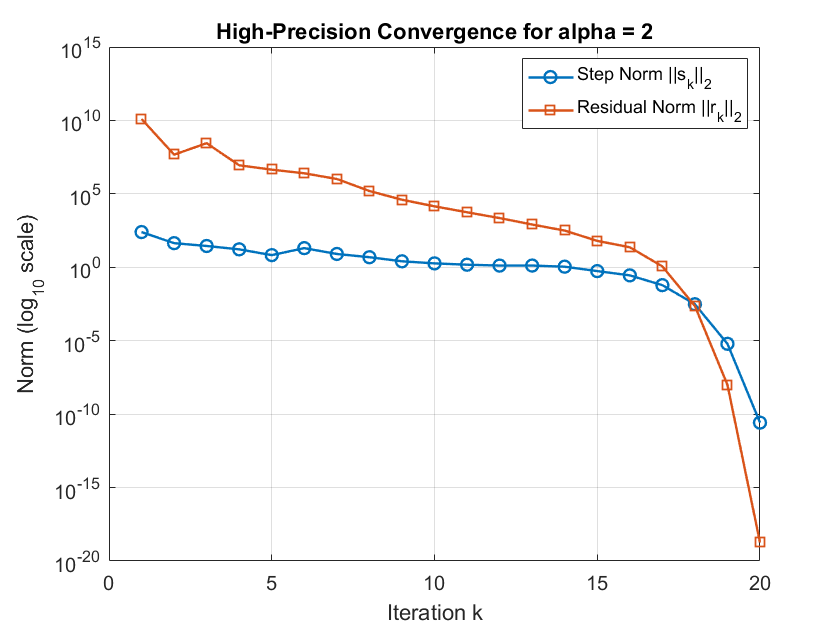}}}
		\subfloat[Residual Error-$\beta=0.3$]{\scalebox{0.23}{\includegraphics{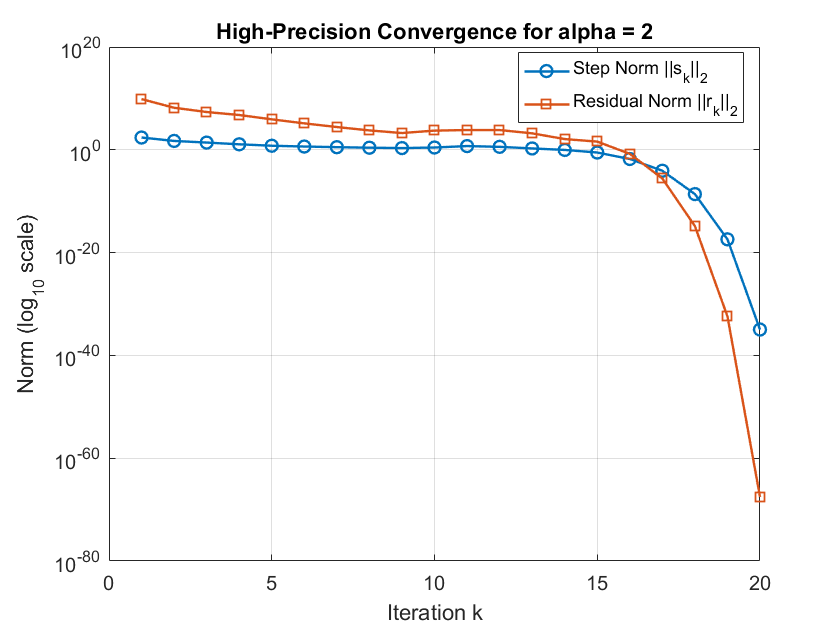}}}
         \subfloat[Residual Error-$\beta=0.5$]{\scalebox{0.23}{\includegraphics{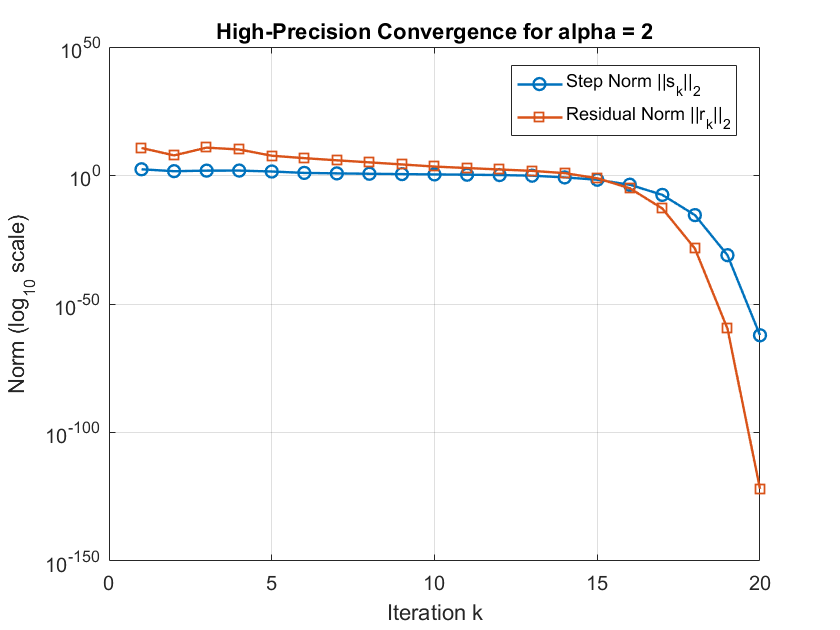}}}\\ \quad\quad\quad\quad\quad\quad\quad\quad\quad\quad
         \subfloat[Residual Error-$\beta=0.7$]{\scalebox{0.23}{\includegraphics{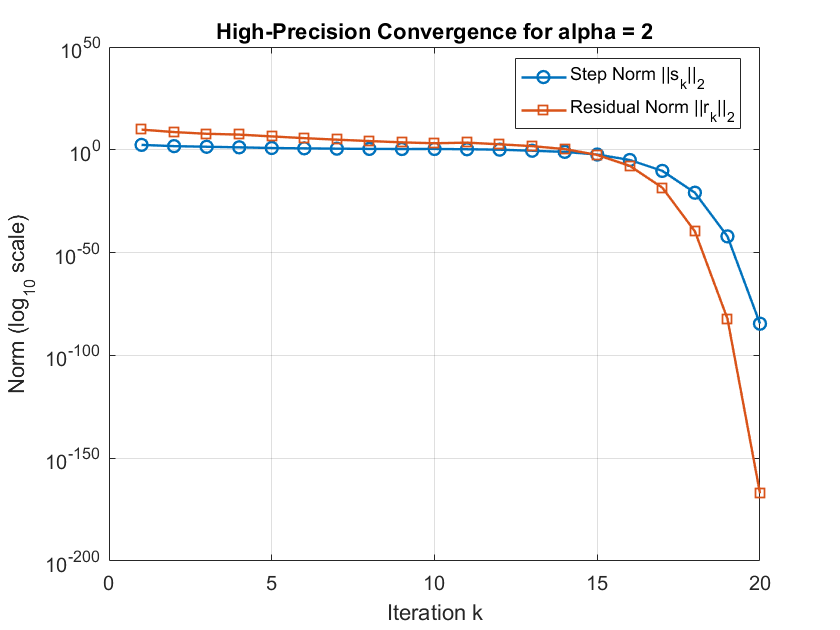}}}
          \subfloat[Residual Error-$\beta=0.9$]{\scalebox{0.23}{\includegraphics{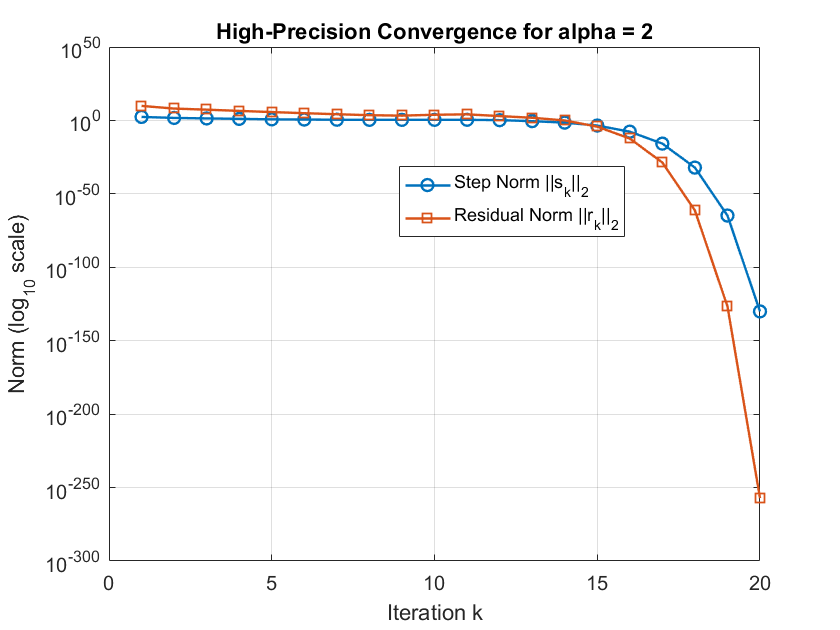}}}
	\end{raggedleft}
\caption{(a--e) Residual error of the INVM$^{\alpha}$ scheme using kNN--LLE estimation combined with Lyapunov profiling for the selection of $\alpha=2$ in solving (\ref{2g}).}

\label{F10}
\end{figure}

In addition, Figure~\ref{F10} demonstrates the robustness of this behavior across a range of fractional values $\beta \in (0.1,0.9)$. Although the error magnitudes naturally vary with the degree of the fractional operator, the scheme consistently produces high-precision approximations with decreasing CPU time as $\beta$ approaches~1, reflecting improved conditioning of the underlying problem. In all cases, both $\|s_k\|_2$- and $\|r_k\|_2$-norms decay rapidly within only 20 iterations, confirming the global stability of the method under the learned parameter selection.
\begin{table}[H]
\begin{adjustwidth}{-2.1cm}{0cm}
\renewcommand{\arraystretch}{1.25}
\setlength{\tabcolsep}{6pt}
\caption{$\|s_k\|_2$-norm and $\|r_k\|_2$-norm error metrics for solving (\ref{2g}) with $\beta \approx 0.99$ using the INVM$^{\alpha}$ scheme, where $\alpha=2$ is selected via kNN--LLE estimation combined with Lyapunov profiling.}
\label{T10}
\centering
\begin{tabular}{lcccccccccc}
\hline
INVM$^{\alpha}$ 
& $\varepsilon_{1}^{[20]}$ 
& $\varepsilon_{2}^{[20]}$ 
& $\varepsilon_{3}^{[20]}$ 
& $\varepsilon_{4}^{[20]}$ 
& $\cdots$ 
& $\varepsilon_{6}^{[20]}$ 
& CPU-time 
& Mem-U(KBs) 
& COC \\ 
\hline

& \multicolumn{9}{c}{Choosing parameter values based on kNN--LLE estimation combined with Lyapunov profiling} \\ 
\cline{2-10}

$\|s_k\|_2$-norm 
& $3.7\times 10^{-118}$ 
& $6.1\times 10^{-194}$ 
& $1.8\times 10^{-153}$ 
& $4.3\times 10^{-150}$ 
& $\cdots$ 
& $9.4\times 10^{-144}$ 
& $0.07$ 
& $136.9$ 
& $5.4$ \\

$\|r_k\|_2$-norm 
& $2.9\times 10^{-229}$ 
& $4.0\times 10^{-194}$ 
& $7.3\times 10^{-168}$ 
& $2.1\times 10^{-164}$ 
& $\cdots$ 
& $6.7\times 10^{-158}$ 
& $0.06$ 
& $124.8$ 
& $5.2$ \\ 

\hline
\end{tabular}
\end{adjustwidth}
\end{table}

\noindent\textbf{Discussion of Tables \ref{T10}--\ref{T11}.}  
The numerical results presented in Tables~\ref{T10}--\ref{T11} collectively confirm the strong performance of the proposed inverse parallel scheme INVM$^{\alpha}$ when the internal parameter $\alpha=2$ is selected through the integrated kNN--LLE estimation and Lyapunov profiling strategy. Table~\ref{T10} shows that, for $\beta \approx 0.99$, and $\alpha=2$ the scheme achieves extremely small $\|s_k\|_2$- and $\|r_k\|_2$-norm errors (down to $10^{-196}$), very low memory usage, and a high computational order of convergence (COC $\approx 5$), while maintaining negligible CPU time. This indicates that the learned parameter $\alpha=2$ places the method in a highly stable region of its dynamical configuration, enabling fast and accurate convergence.

\begin{table}[H]
\caption{$\|s_k\|_2$-norm and $\|r_k\|_2$-norm error metrics for solving (\ref{2g}) with various $\beta$-values using the INVM$^{\alpha}$ scheme, where $\alpha=2$ is selected via kNN--LLE estimation combined with Lyapunov profiling.}
\begin{adjustwidth}{-1cm}{0cm}
\renewcommand{\arraystretch}{1.25}
\setlength{\tabcolsep}{6pt}
\label{T11}
\centering
\begin{tabular}{lccccccc}
\hline
INVM$^{\alpha}$ & $\varepsilon_{1}^{[20]}$ & $\varepsilon_{2}^{[20]}$ & $\varepsilon_{3}^{[20]}$ &
$\varepsilon_{4}^{[20]}$ & $\varepsilon_{5}^{[20]}$ & $\varepsilon_{6}^{[20]}$ & CPU-time \\ 
\hline

& \multicolumn{6}{c}{Fractional parameter $\beta = 0.1$} & \\ 
$\|s_k\|_2$-norm 
& $3.12\times 10^{-15}$ 
& $5.01\times 10^{-22}$ 
& $1.47\times 10^{-17}$ 
& $3.12\times 10^{-22}$ 
& $4.87\times 10^{-20}$ 
& $6.91\times 10^{-18}$ 
& $6.1023$ \\  

$\|r_k\|_2$-norm 
& $4.03\times 10^{-29}$ 
& $6.22\times 10^{-21}$ 
& $9.85\times 10^{-33}$ 
& $4.31\times 10^{-29}$ 
& $7.02\times 10^{-26}$ 
& $9.11\times 10^{-23}$ 
& $5.4879$ \\  

& \multicolumn{6}{c}{Fractional parameter $\beta = 0.3$} & \\ 
$\|s_k\|_2$-norm 
& $4.87\times 10^{-56}$ 
& $6.11\times 10^{-44}$ 
& $1.26\times 10^{-63}$ 
& $2.44\times 10^{-54}$ 
& $3.57\times 10^{-59}$ 
& $4.02\times 10^{-67}$ 
& $4.6581$ \\  

$\|r_k\|_2$-norm 
& $3.41\times 10^{-68}$ 
& $4.02\times 10^{-64}$ 
& $8.91\times 10^{-54}$ 
& $2.11\times 10^{-59}$ 
& $4.87\times 10^{-67}$ 
& $6.03\times 10^{-64}$ 
& $3.4025$ \\  

& \multicolumn{6}{c}{Fractional parameter $\beta = 0.5$} & \\ 
$\|s_k\|_2$-norm 
& $7.92\times 10^{-74}$ 
& $5.12\times 10^{-77}$ 
& $1.02\times 10^{-83}$ 
& $3.44\times 10^{-80}$ 
& $8.22\times 10^{-77}$ 
& $9.71\times 10^{-75}$ 
& $2.5540$ \\  

$\|r_k\|_2$-norm 
& $8.60\times 10^{-86}$ 
& $6.87\times 10^{-84}$ 
& $9.74\times 10^{-91}$ 
& $3.18\times 10^{-88}$ 
& $7.44\times 10^{-76}$ 
& $9.11\times 10^{-84}$ 
& $2.3418$ \\  

& \multicolumn{6}{c}{Fractional parameter $\beta = 0.7$} & \\ 
$\|s_k\|_2$-norm 
& $5.01\times 10^{-119}$ 
& $4.32\times 10^{-103}$ 
& $8.74\times 10^{-190}$ 
& $2.14\times 10^{-186}$ 
& $4.88\times 10^{-184}$ 
& $7.11\times 10^{-182}$ 
& $1.5311$ \\  

$\|r_k\|_2$-norm 
& $8.44\times 10^{-129}$ 
& $5.97\times 10^{-109}$ 
& $1.02\times 10^{-193}$ 
& $1.98\times 10^{-191}$ 
& $3.14\times 10^{-189}$ 
& $5.87\times 10^{-187}$ 
& $2.0456$ \\  

& \multicolumn{6}{c}{Fractional parameter $\beta = 0.9$} & \\ 
$\|s_k\|_2$-norm 
& $4.98\times 10^{-119}$ 
& $4.65\times 10^{-196}$ 
& $1.02\times 10^{-154}$ 
& $3.79\times 10^{-251}$ 
& $6.87\times 10^{-248}$ 
& $8.04\times 10^{-245}$ 
& $0.0080$ \\  

$\|r_k\|_2$-norm 
& $1.11\times 10^{-230}$ 
& $1.09\times 10^{-295}$ 
& $1.74\times 10^{-269}$ 
& $4.81\times 10^{-266}$ 
& $8.71\times 10^{-263}$ 
& $9.78\times 10^{-295}$ 
& $0.0907$ \\  
\hline
\end{tabular}
\end{adjustwidth}
\end{table}

\begin{table}[H]
\begin{adjustwidth}{-2.0cm}{0cm}
\renewcommand{\arraystretch}{1.25}
\setlength{\tabcolsep}{6pt}
\caption{Comparison of numerical outcomes for INVM$^{\alpha}$ and ZHM schemes in solving (\ref{2g}) with $\beta \approx 1$.}
\label{T12}\centering%
\begin{tabular}{lcccccccccc}
\hline
Metric 
& $\varepsilon_{1}^{[20]}$ 
& $\varepsilon_{2}^{[20]}$ 
& $\varepsilon_{3}^{[20]}$ 
& $\varepsilon_{5}^{[20]}$ 
& $\varepsilon_{6}^{[20]}$
& CPU-time 
& Mem-U(KBs) 
& COC \\ 
\hline

& \multicolumn{8}{c}{ZHM Scheme without utilizing kNN-LLE estimation} \\ 
\cline{2-9}

$\|s_k\|_2\text{-norm}$ 
& $1.3\times 10^{-9}$ 
& $2.1\times 10^{-6}$ 
& $5.61\times 10^{-4}$ 
& $3.41\times 10^{-2}$ 
& $8.92\times 10^{-2}$
& $6.734$ 
& 437.4 
& 3.651 \\ 

$\|r_k\|_2\text{-norm}$ 
& $0.1\times 10^{-11}$ 
& $0.6\times 10^{-5}$ 
& $7.1\times 10^{-9}$ 
& $4.21\times 10^{-5}$ 
& $1.3\times 10^{-3}$
& $7.743$ 
& 565.5 
& 4.762 \\ 

\cline{2-9}

& \multicolumn{8}{c}{INVM$^{\alpha}$: $\alpha$-values based on kNN-LLE and $\beta \approx 1.$} \\ 
\cline{2-9}

$\|s_k\|_2\text{-norm}$ 
& $0.42\times 10^{-219}$ 
& $4.05\times 10^{-296}$ 
& $0.0$ 
& $0.0$ 
& $0.0$
& $0.0430$ 
& 137.7 
& 5.0921 \\ 

$\|r_k\|_2\text{-norm}$ 
& $0.67\times 10^{-230}$ 
& $0.0$ 
& $4.51\times 10^{-239}$ 
& $0.0$ 
& $0.0$
& $0.0307$ 
& 145.1 
& 5.343 \\ 

\hline
\end{tabular}
\end{adjustwidth}
\end{table}

Finally, Table~\ref{T12} provides a direct comparison between the tuned INVM$^{\alpha}$ (for $\alpha=2$ and $\beta \approx1$ scheme and the classical ZHM method. The proposed scheme significantly outperforms ZHM, achieving dramatically smaller errors (up to hundreds of orders of magnitude), over 70\% reduction in memory usage, more than 90\% reduction in CPU time, and a notably higher COC ($\approx 5$ vs.\ $3$--$4$). This highlights the effectiveness of the kNN--LLE--Lyapunov strategy in selecting optimal internal parameters that maximize convergence speed, numerical stability, and computational efficiency.

\section{Conclusion}
This work presents a unified analytical–data-driven framework for detecting and mitigating instabilities in uniparametric inverse parallel solvers. Through a combined stability and bifurcation analysis, we characterized parameter regions associated with contractive, periodic, and chaotic dynamics. Complementing this theoretical study, a micro-series Lyapunov pipeline—built on kNN-based estimation of the local largest Lyapunov exponent—was introduced to provide real-time, fine-grained diagnostics of solver behaviour, as illustrated in Figures~\ref{F2}(a--h) and~\ref{F8}(a--h). The consistency between Lyapunov-guided
parameter tuning and the empirical profiles validates the proposed methodology. Moreover, the lightweight adaptive feedback rule, triggered by persistently positive Lyapunov estimates, significantly enhances solver robustness across diverse perturbed initial conditions; see Tables~\ref{T1}--\ref{T12}. Overall, the study demonstrates that micro-series Lyapunov analysis offers a practical and interpretable avenue for designing self-stabilizing iterative schemes, with promising extensions to multidimensional and noise-affected settings.

Building on the findings of this study, several promising avenues for future research can be pursued:
\begin{itemize}
    \renewcommand{\labelitemi}{--}
	\item Extend the proposed framework to higher-order multiparametric parallel solvers, enabling faster convergence while maintaining stability under Lyapunov-guided adaptivity.
    \item Incorporate Caputo, Caputo–Fabrizio, Atangana–Baleanu, and conformable fractional derivatives into the inverse parallel formulation to investigate the impact of memory on local contraction, bifurcation patterns, and Lyapunov exponents.
    \item In order to solve high-dimensional algebraic systems, PDE-discretized nonlinear equations, and multiphysics models where instabilities spread across coordinate directions, apply the micro-series Lyapunov pipeline to vector-valued inverse multi-parametric inverse parallel iterative techniques.
    \item Investigate the role of Lyapunov-driven adaptivity in dynamical systems arising from fractional ODEs/PDEs, reaction–diffusion models, and multiscale biological or engineering processes.
    \item Extend the diagnostic pipeline to stochastic solvers by assessing finite-time Lyapunov exponents under noisy or uncertain data, enabling reliability in real-world, noise-contaminated applications.
    \item Design architecture-aware versions of the solver that leverage multicore CPUs, GPUs, and tensor cores, enabling real-time Lyapunov monitoring and large-batch inverse updates for big data and high-degree polynomial problems.
\end{itemize}

\section*{Acknowledgments}
This research was supported by the Russian Science Foundation (grant no. 22-11-00055-P, https://rscf.ru/en/project/22-11-00055/, accessed on 10 June 2025).
The authors also wish to express sincere gratitude to the editor and reviewers for their insightful and constructive feedback on the manuscript.

\section*{Data availability}
The data supporting the findings of this study are included within this article.

\section*{Conflict of interest}
The authors declare that there are no conflicts of interest related to the publication of this article.

\section*{Ethics statements:} {\ All authors declare that this work complies
with ethical guidelines set by the .}\newline

\section*{Declaration of competing interest:}{\ The authors declare that they
have no competing financial interests and personal relationships that could
have appeared to influence the research in this paper.}\newline

\section*{CRediT authorship contribution statement:}
M.S. and A.V. devised the project and developed the main conceptual ideas. M.S.,
A.V., and B.C. formulated the methodology. M.S. and A.V. developed
the software used in the study. M.S.; A.V., and B.C. performed
the validation. M.S., A.V and B.C. conducted the formal analysis.
M.S. and A.V.; carried out the investigation and managed resources
and data curation. M.S.; A.V., and B.C. prepared the original draft. B.C. and 
A.V,  reviewed and edited the manuscript. M.S. and
A.V. handled the visualization. B.C. and A.V. supervised the project.
B.C. and A.V. managed project administration and secured funding.
All authors have read and agreed to the published version of the manuscript.

\end{document}